\documentclass[11pt,a4paper]{amsart}
\usepackage[utf8]{inputenc}
\usepackage[english]{babel}
\usepackage[left=2.5cm,right=2.5cm,top=2cm,bottom=2cm]{geometry}
\usepackage{graphicx}
\usepackage{amssymb}
\usepackage{amsmath}
\usepackage{amsthm}
\usepackage{algorithm}
\usepackage{algorithmic}
\usepackage{comment}
\usepackage{pdfpages}
\usepackage{appendix}
\usepackage{stmaryrd}
\usepackage{tikz}
\usepackage{mathtools}
\usepackage{bbm}
\definecolor{bordeau}{rgb}{0.5,0,0}
\definecolor{pslblue}{RGB}{36, 56, 141}
\usepackage[maxbibnames=50, style=alphabetic]{biblatex}

\usepackage{csquotes}
\usepackage{hyperref}
\hypersetup{
  colorlinks   = true, 
  linkcolor    = pslblue,
  citecolor    = bordeau, 
  urlcolor = pslblue      
}

\providecommand{\E}{\mathbb E}
\providecommand{\ed}{\mathrm e}

\providecommand{\proba}{\mathbb P}
\providecommand{\N}{\mathbb N}

\providecommand{\R}{\mathbb R}
\providecommand{\D}{\mathbb D}
\providecommand{\C}{\mathbb C}

\newcommand{\ddiff}{\mathrm d}

\numberwithin{equation}{section}
\newtheorem{theorem}{Theorem}[section]
\newtheorem{proposition}[theorem]{Proposition}
\newtheorem{corollary}[theorem]{Corollary}
\newtheorem{lemma}[theorem]{Lemma}
\theoremstyle{definition}
\newtheorem{definition}[theorem]{Definition}

\newtheorem{remark}[theorem]{Remark}

\title[Asymptotic analysis of the characteristic polynomial]
{Asymptotic analysis of the characteristic 
polynomial for the Elliptic Ginibre Ensemble}
\author{Quentin François, David Garc\'ia-Zelada}

\address
{Quentin François: CEREMADE, CNRS, Université Paris-Dauphine, Université PSL, 75016 Paris, France
\& DMA, École normale supérieure, Université PSL, CNRS, 75005 Paris, France.}
\email{\href{mailto:quentin.francois@dauphine.psl.eu}{quentin.francois@dauphine.psl.eu}}

\address{
David Garc\'ia-Zelada: Laboratoire de Probabilités, Statistique et Modélisation, 
UMR CNRS 8001, 
Sorbonne Université, 4 Place Jussieu, 75005 Paris, France.}
\email{\href{mailto:david.garcia-zelada@sorbonne-universite.fr}{david.garcia-zelada@sorbonne-universite.fr}}

\begin{document}

\begin{abstract}
We consider the complex Elliptic Ginibre Ensemble,
a family of random matrix models introduced by Girko that interpolates between
the Ginibre Ensemble and the Gaussian Unitary Ensemble
and such that its empirical spectral measure converges
to the uniform measure on an ellipse. We show
the convergence in law of its normalised characteristic polynomial 
outside of this ellipse. Our proof contains two main steps.
We first show the tightness of the normalised characteristic polynomial 
using the link between the Elliptic Ginibre Ensemble and Hermite polynomials. 
This part relies on the uniform control of the Hermite kernel which is derived 
from the recent work of Akemann, Duits and Molag. 
In the second step, we identify the limiting object as the exponential of a 
Gaussian analytic function. 
The limit expression is derived from the 
convergence of traces of Chebyshev polynomials of 
random matrices by the method of moments. 
These traces of Chebyshev polynomials appear naturally 
as a kind of centered version, or \emph{normal ordering}, 
of the traces of the monomials.
This work answers the interpolation problem raised in the work of Bordenave, Chafaï 
and the second author of this paper for the integrable case of the Elliptic Ginibre Ensemble 
and is therefore a fist step towards the conjectured universality of this result.
\end{abstract}

\maketitle
\section{Introduction and main result}

\subsection{The model of the Elliptic Ginibre Ensemble (EGE)}

The random matrices that we consider in this paper are sampled from the complex 
 Elliptic Ginibre Ensemble introduced by Girko in \cite{Girko_elliptic}. 
This model is parameterized by $t \in [0,1]$ and interpolates between the Ginibre Ensemble 
and the Gaussian Unitary Ensemble (GUE) for $t=0$ and $t=1$ respectively. 
A concise review of this model can be found in \cite{Khoruzhenko_Sommers}.
Its  law is the one
of a random matrix given by the following construction.
\\
\\
Recall that a matrix sampled from the Gaussian Unitary Ensemble
is a Hermitian random matrix whose density is proportional to
$e^{-\mathrm{Tr}(M^2)/2}$.
Consider $X_n$ and $Y_n$ independent random matrices sampled 
from the Gaussian Unitary Ensemble of size $n \geq 1$.
The law of the Elliptic Ginibre Ensemble at $t \in [0,1]$ is the law of the matrix
\begin{equation}
\label{eq:elliptic_sum_gue}
    A_{n,t} = \sqrt{\frac{1+t}{2}}X_n + i \sqrt{\frac{1-t}{2}}Y_n, 
\end{equation} 
where $i$ is the imaginary unit. Equivalently,  $A_{n,t}$ 
has a law  proportional to
\begin{equation}
    \exp \left( - \frac{1}{1-t^2} \mathrm{Tr} \left[M^*M - \frac{t}{2} (M^2 + (M^*)^2) \right] \right) \ddiff M,
\end{equation}
where $\ddiff M = \prod_{1 \leq i,j \leq n} \ddiff M_{ij}$ is the product 
Lebesgue measure on the entries of the matrix, see \cite[Eq. (4)]{Akemann_Duits_Molag}. 
Notice that $A_{n,t}$ could also be defined as a
	centered complex Gaussian
	matrix whose entries $a_{ij}$ satisfy 
	$\E[|a_{ij}|^2] = 1$, 
	$\E[a_{ij}a_{ji}] = t$ for every $i,j$, and for $i \neq j$
	$\E[a_{ij} \bar a_{ji}] = \E[a_{ij}^2] = 0$
	while covariance between $a_{ij}$ and $a_{i'j'}$
	are zero if $\{i,j\} \neq \{i',j'\}$. Moreover, 
	for $i\neq j$, $ \E[(a_{ij}a_{ji} - t)^2] = t^2 $ 
	and $\E[|a_{ij} a_{ji} - t|^2] = 1$.
Many results are known for EGE matrices. 
In particular, the limiting eigenvalue distribution has been proved 
by Girko
to be the uniform law on the ellipse 
centered at the origin with half long axis $1+t$ and short axis $1-t$. 
We refer to  \cite[Theorem 7]{Girko_elliptic} and \cite{Sommers_al} for the first instances
 of this result.\\
\\
In the recent work \cite{Bordenave_Chafai_Garcia}, it has been proved that the spectral radius of matrices with i.i.d.\ centered entries, 
called Girko matrices, converges in probability to $1$ under the minimal assumption of a second moment on its entries. 
In order to derive this result, the authors considered the reciprocal characteristic polynomial associated to 
such matrices defined by $q_n(z) = z^n p_n \left(\frac{1}{z} \right)$ for $z \in \D = \{ y \in \C: |y| < 1 \}$, 
where $p_n$ is the characteristic polynomial. The main result of \cite{Bordenave_Chafai_Garcia} is the convergence in law, for the topology of 
local uniform convergence, of the sequence of functions $\{q_n\}_{n\geq 1}$ to a random function which 
is universal, in the sense that its expression involves only the second moment of the entries of the matrix. 
Our result aims at deriving the convergence of the normalised characteristic polynomial in the case of the 
EGE \eqref{eq:elliptic_sum_gue} at each $t \in [0,1]$ and at identifying the limiting object in the conjectured universality.
In particular, for $t=1$ our result gives the convergence of the characteristic polynomial for 
GUE matrices to a random analytic function.\\
\\
\noindent
Characteristic polynomials of random matrices have been studied extensively. 
For Haar unitary matrices or, more generally, for Circular $\beta$-Ensembles 
(C$\beta$E), 
the characteristic polynomial
outside the unit disk
behaves in a similar way as in
Theorem \ref{theorem:main_result} for $t=0$ as is stated in
\cite[Theorem 1.3]{Najnudel_Paquette_Simm}.
Moreover, the characteristic polynomial
inside and outside the unit disk exhibit the same but independent limiting behavior.
More interestingly, the scaling limit
around a point at the unit circle has been studied in \cite{Chhaibi_2016} 
by showing a convergence towards 
a random analytic function whose zeros form a determinantal point process on the real line. 
Limit expressions for the characteristic polynomial of C$\beta$E matrices 
are furthermore related to 
the Gaussian multiplicative chaos and to the theory of orthogonal polynomials on the unit circle, 
see \cite{Lambert_Najnudel_cbe}.
In the case of Gaussian 
$\beta$-Ensembles, approximations of the characteristic polynomial in the complex plane 
where found in terms of 
log-correlated Gaussian fields, see \cite{Lambert_Paquette}. 
For Haar random matrices, asymptotics for moments of derivatives of 
the characteristic polynomial have also been computed in \cite{Simm_Fei}. 
They derive limits both inside the unit disk and for 
mesoscopic and microscopic regimes when $z$ approaches the unit circle.
The cases of orthogonal, 
symplectic and GUE random matrices have been studied in \cite{limit_charac_poly_classical_ensembles} 
where ratios of characteristic polynomials are shown to converge to a random entire function which was 
constructed in \cite{Chhaibi_2016} and related to Haar random matrices. 
\\
\\
The study of the reciprocal characteristic polynomial for Girko matrices in 
\cite{Bordenave_Chafai_Garcia} was partially inspired from the work \cite{Basak_Zeitouni} 
on Toeplitz matrices. The same object was studied for other models. The case of 
sparse matrix models having i.i.d.\ non-centered Bernoulli entries was treated 
in \cite{Coste}. The reciprocal characteristic polynomial of such matrices converges 
to a random function expressed using Poisson series \cite{Coste}. 
In \cite{Coste_Lambert_Zhu_permutations}, the same type of convergence was obtained 
for sums of random uniform permutation matrices. For a fixed number of random matrices 
in the sum, the limit has the same form given by the exponential of a Poisson series, whereas 
for a number of terms going to infinity in a prescribed way, the limit has the 
form given by the exponential of a Gaussian series as in \cite{Bordenave_Chafai_Garcia}.
Exponential of Poisson series were also obtained as the limit of characteristic polynomials 
for permutation matrices in \cite{Bahier_1, Bahier_2}.
\\
\\
In relation to elliptic matrices, the uniform law on the 
ellipse can be obtained as the asymptotic distribution of zeroes of random polynomials which are related 
to Weyl polynomials, see \cite{zeroes_random_poly_2024}.\\
\\
The motivation from the work \cite{Bordenave_Chafai_Garcia} 
was to obtain the convergence of the spectral radius for Girko matrices. 
One could ask for a study of the fluctuations around the limit. 
For the Ginibre Ensemble, 
one has Gumbel fluctuations for the maximum modulus around $1$, 
see \cite{Rider}. The Gumbel distribution also appears as the limit 
fluctuation for the largest real part of either real of complex Ginibre matrices 
\cite{Cipolloni_al}. For the GUE, 
one has Tracy-Widom fluctuations for the maximum eigenvalue around $2$, 
see \cite{Tracy_Widom} and references therein. 
In \cite{Johansson}, we may find a family of determinantal processes that 
interpolates between a Poisson process with intensity $\ed^{-x}$ 
and the Airy process. The distribution 
function of its last particle interpolates between the Gumbel and 
Tracy-Widom distributions, see 
\cite[Theorem 1.3]{Johansson}. As a two-dimensional version, \cite{Bender} 
considered the Elliptic Ginibre Ensemble and 
an interpolating determinantal processes to prove scaling limits 
for the eigenvalue point process.

\subsection{Main result}

Let $n \geq 1$ and $t \in [0,1]$. Consider 
$p_{n,t}(z) = \det (z - \frac{1}{\sqrt{n}}A_{n,t})$, the scaled characteristic polynomial 
of a matrix $A_{n,t}$ sampled from the Elliptic Ginibre Ensemble 
\eqref{eq:elliptic_sum_gue}.  
Define 
 $f_{n,t}: \D \to \mathbb C$ as the normalised characteristic polynomial 
 of $A_{n,t}$,

\begin{equation}
f_{n,t}(z) \coloneqq \det \left(1 + tz^2 
- \frac{z}{\sqrt{n}}A_{n,t} \right) \ed^{-\frac{ntz^2}{2}}. \label{def:f_n_t}
\end{equation}

\noindent
We endow the space of holomorphic functions
on $\mathbb D$ with the topology
of uniform convergence on compact sets and state our main result as follows.
 
 \begin{theorem}[Convergence of the normalised characteristic polynomial]
 \label{theorem:main_result}
    As $n \rightarrow \infty$, 
    \begin{equation*}
      f_{n,t}  \xrightarrow[]{\mathrm{law}}  \exp(-F_t)
    \end{equation*}
     where $F_t$ is the Gaussian holomorphic function on $\mathbb D$ defined by
    \begin{equation}
    \label{eq:F_t}
        F_t(z) \coloneqq \sum_{k \geq 1} X_{k} \frac{z^k}{\sqrt{k}}
    \end{equation}
     for a family $( X_{k} )_{k \geq 1}$ of independent Gaussian random variables 
     on $\mathbb C$
     satisfying 
	 \begin{equation*}
        \E[X_{k}] = 0, \ \E[X_{k}^2] = t^k \text{ and } \E[|X_{k}|^2] = 1.
      \end{equation*}
\end{theorem}

\vspace{2mm}
\noindent
Let us give some intuition for the choice in \eqref{def:f_n_t}. 
Since the empirical measure of eigenvalues
of $A_{n,t}$ converges to the uniform measure $\sigma_t$ on the ellipse
\begin{equation}
	\label{eq:ellipse}
	\mathcal E_t = \left\{ x+iy \in \mathbb C: 
	\left( \frac{x}{1+t} \right)^2 + \left( \frac{y}{1-t} \right)^2 \leq 1 \right\},
\end{equation}
to study the behavior of
the characteristic polynomial $p_{n,t}$ on
$\mathbb C \setminus \mathcal E_t$, we send this set 
to the unit disk in the simplest holomorphic way, namely, by using the map
$g_t: \D \setminus \{0\} \to \mathbb C \setminus \mathcal E_t$,
\[g_t(z) = \frac{1}{z} + tz.\]
In the case of $t=1$, we should define $\mathcal E_1 = [-2, 2]$ 
and $\sigma_1 = \lim_{t \to 1} \sigma_t $ is the semi-circular law,
which is consistent with Wigner's semicircular law. In this case,
$g_1$ is the so-called Joukowsky transform and,
moreover, 
we have the simple relation
$g_t(z) = \sqrt tg_1(\sqrt tz)$ whenever $t \neq 0$. 
Under this change of variables, the characteristic polynomial 
$p_{n,t}$ is
\[p_{n,t}\circ g_t(z) = \det \left(g_t(z) - \frac{A_{n,t}}{\sqrt n} \right)
= \frac{1}{z^n}\det \left(1 + tz^2 - \frac{z}{\sqrt{n}}A_{n,t} \right)
\]
Finally, 
$n^{-1}\log \det \left(1 + tz^2 - z n^{-1/2}A_{n,t} \right)$
should converge to
$\int \log(1+tz^2-zw) \mathrm d\sigma_t(w)$
which is $tz^2/2$, as seen in the proof of Lemma
\ref{prop:expectation_convergence}, suggesting the exponential factor 
in $f_{n,t}$. \\
\\
 Using these notations, from Theorem \ref{theorem:main_result}
 we obtain the convergence of the
normalised characteristic polynomial
$\tilde p_{n,t}(u) = (g_t^{-1}(u))^n \ed^{-nt (g_t^{-1}(u))^2/2}  p_{n,t}(u)$
\[ \tilde p_{n,t} \xrightarrow[]{\mathrm{law}}
\exp(-F_t \circ g_t^{-1}) \]
for the topology of uniform convergence
on compact sets of 
$\mathbb C \setminus \mathcal E_t$.
This is, in fact, equivalent 
to Theorem \ref{theorem:main_result} due to the holomorphicity
of $f_{n,t}$ at zero.
It explains the notation ``normalised characteristic polynomial'' since $f_{n,t}$ and $\tilde p_{n,t}$
are the same written in different coordinate systems.\\
\\
From Theorem \ref{theorem:main_result}, one can derive the 
following result given in \cite[Theorem 2.2]{ORourke_Renfrew} 
for a class of elliptic matrices that includes our Gaussian case. 
Nevertheless, since an
explicit density can be written for the eigenvalues in the Gaussian case, 
we may also use large deviation arguments
to obtain the lack of outliers.\\

\begin{corollary}[Lack of outliers]
	Let $C \subset \C$ be a closed set disjoint from $\mathcal E_t$. Then, 
	as $n \to \infty$,
	\begin{equation}
	  N_n(C) \coloneqq \# \left\{i \in [n]: \frac{\lambda_i}{\sqrt{n}} \in C \right\} 
	  \xrightarrow[]{\phantom{c} \proba \phantom{c}}  0.
	\end{equation}
\end{corollary}

\begin{proof}
Let $C \subset \mathbb C$ be a closed set disjoint from $\mathcal E_t$.
Recall that $g_t(z) = \frac{1}{z} + tz$ and consider
$\tilde C = g_t^{-1}(C)$ which is closed 
in $\mathbb D \setminus \{0\}$ so that its closure $K$
on $\mathbb D$
is compact.  Then, 
\begin{align*}
    \proba[|N_n(C)| >0]
    &= \proba[\inf_{z \in C} |p_{n,t}(z)| = 0] \\
    &= \proba[\inf_{u \in K } |f_{n,t}(u)| = 0] \\
    & \rightarrow \proba[\inf_{ u \in K} |e^{-F_t(u)}| = 0] = 0.
\end{align*}
\end{proof}

\begin{figure}[ht] 
  \begin{minipage}[b]{0.5\linewidth}
    \centering
    \includegraphics[scale=0.23]{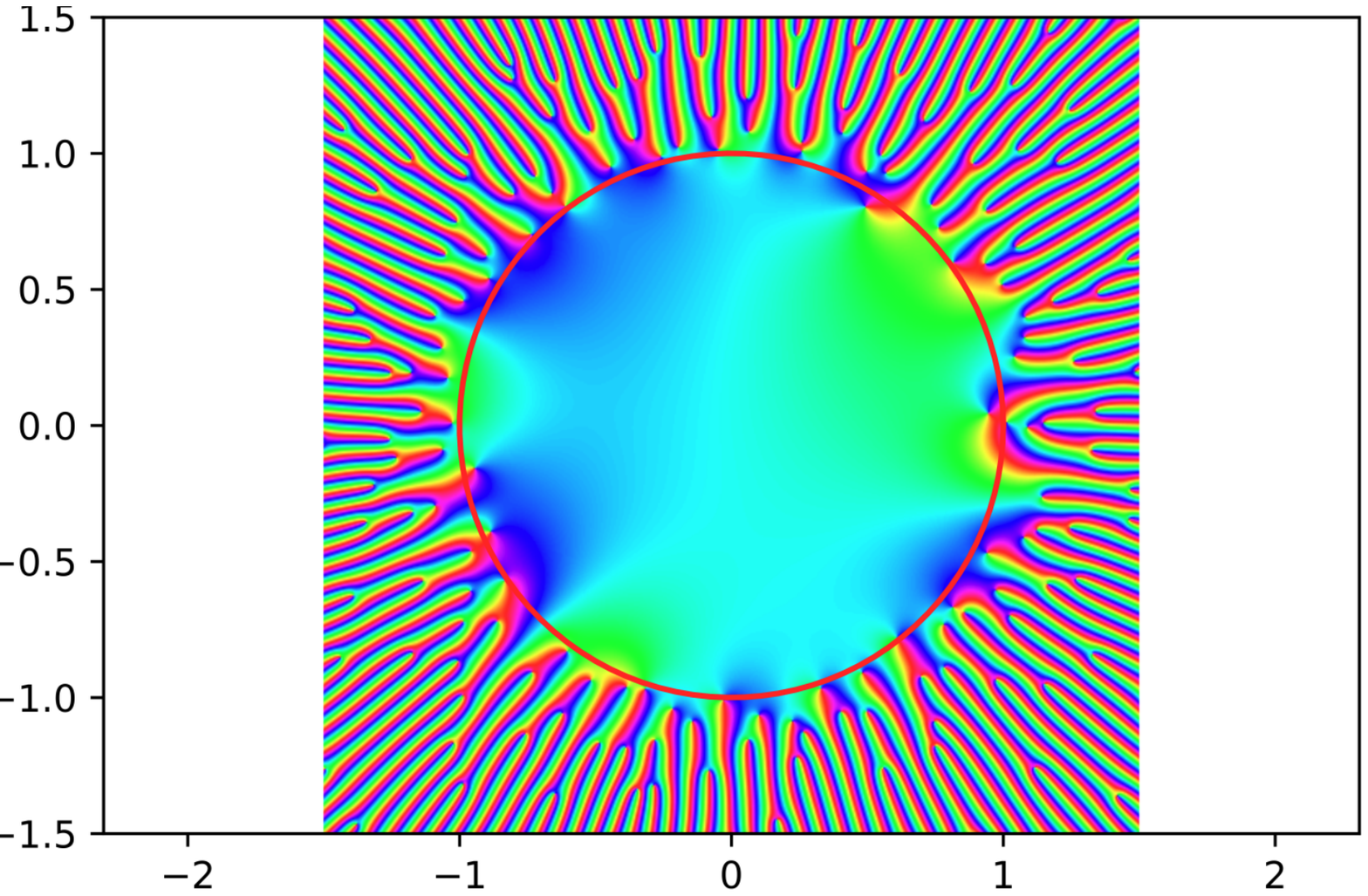} 

  \end{minipage}
  \begin{minipage}[b]{0.5\linewidth}
    \centering
    \includegraphics[scale=0.23]{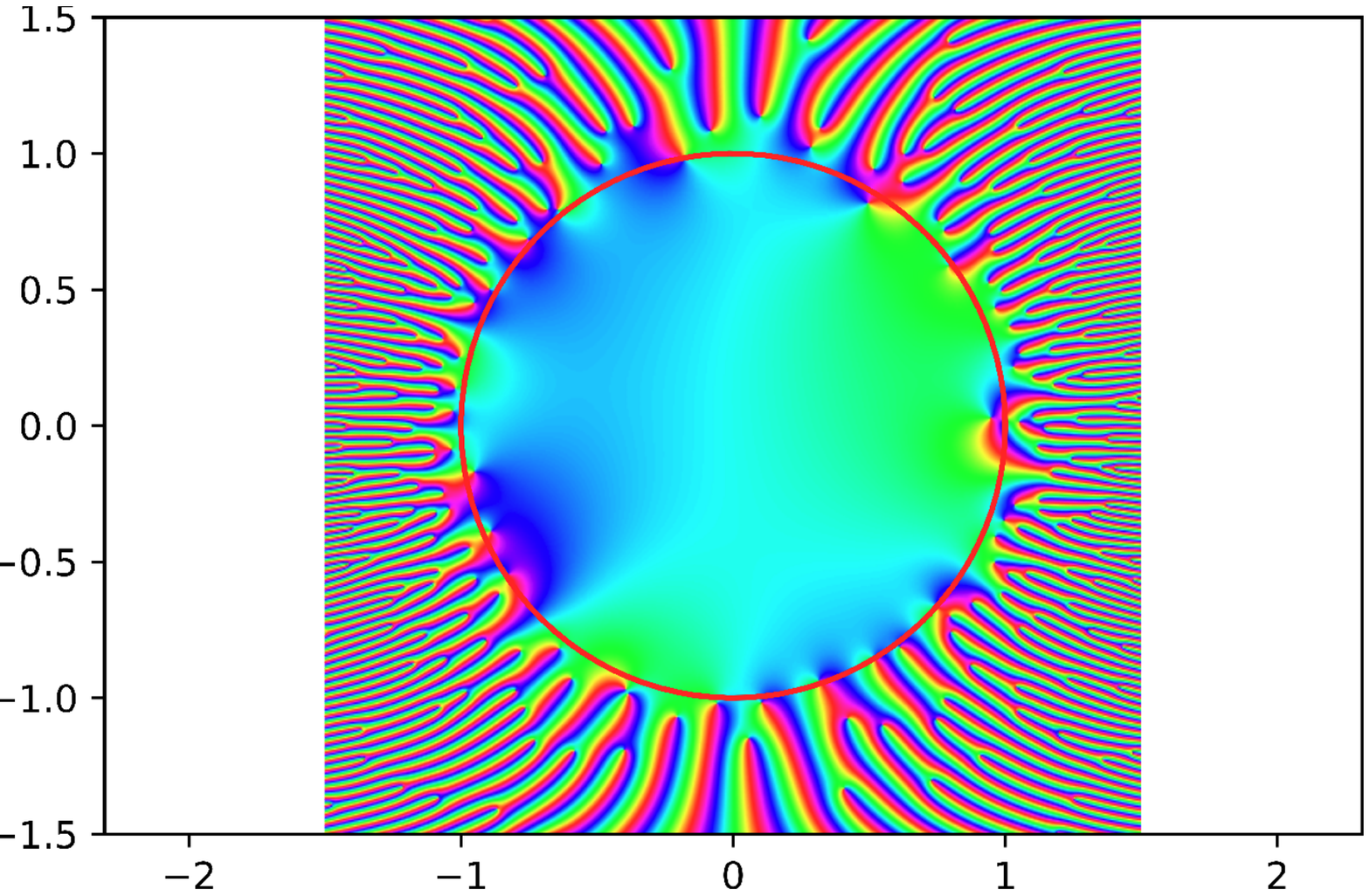} 

  \end{minipage} 
  \begin{minipage}[b]{0.5\linewidth}
    \centering
    \includegraphics[scale=0.23]{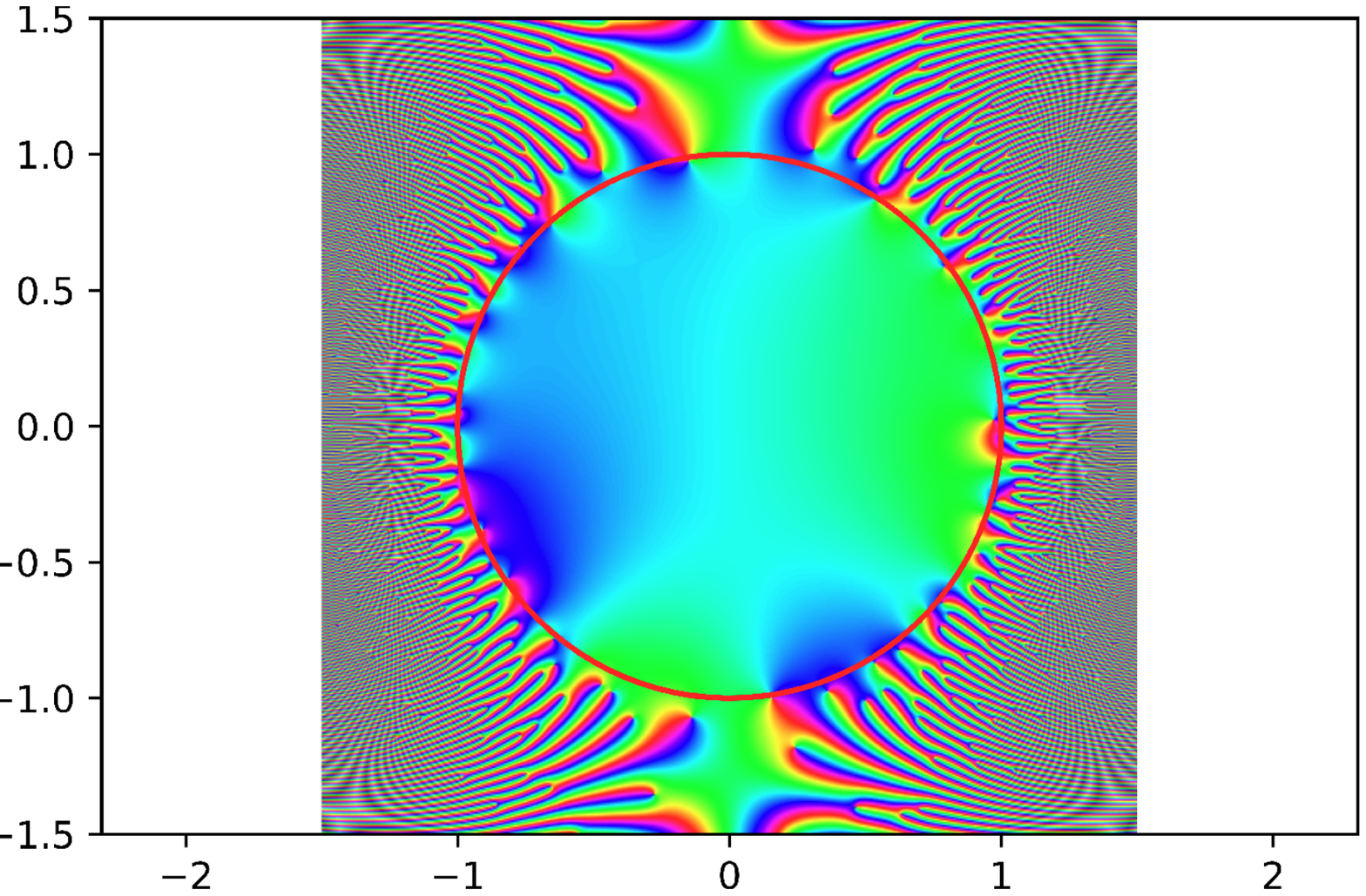} 

  \end{minipage}
  \begin{minipage}[b]{0.5\linewidth}
    \centering
    \includegraphics[scale=0.23]{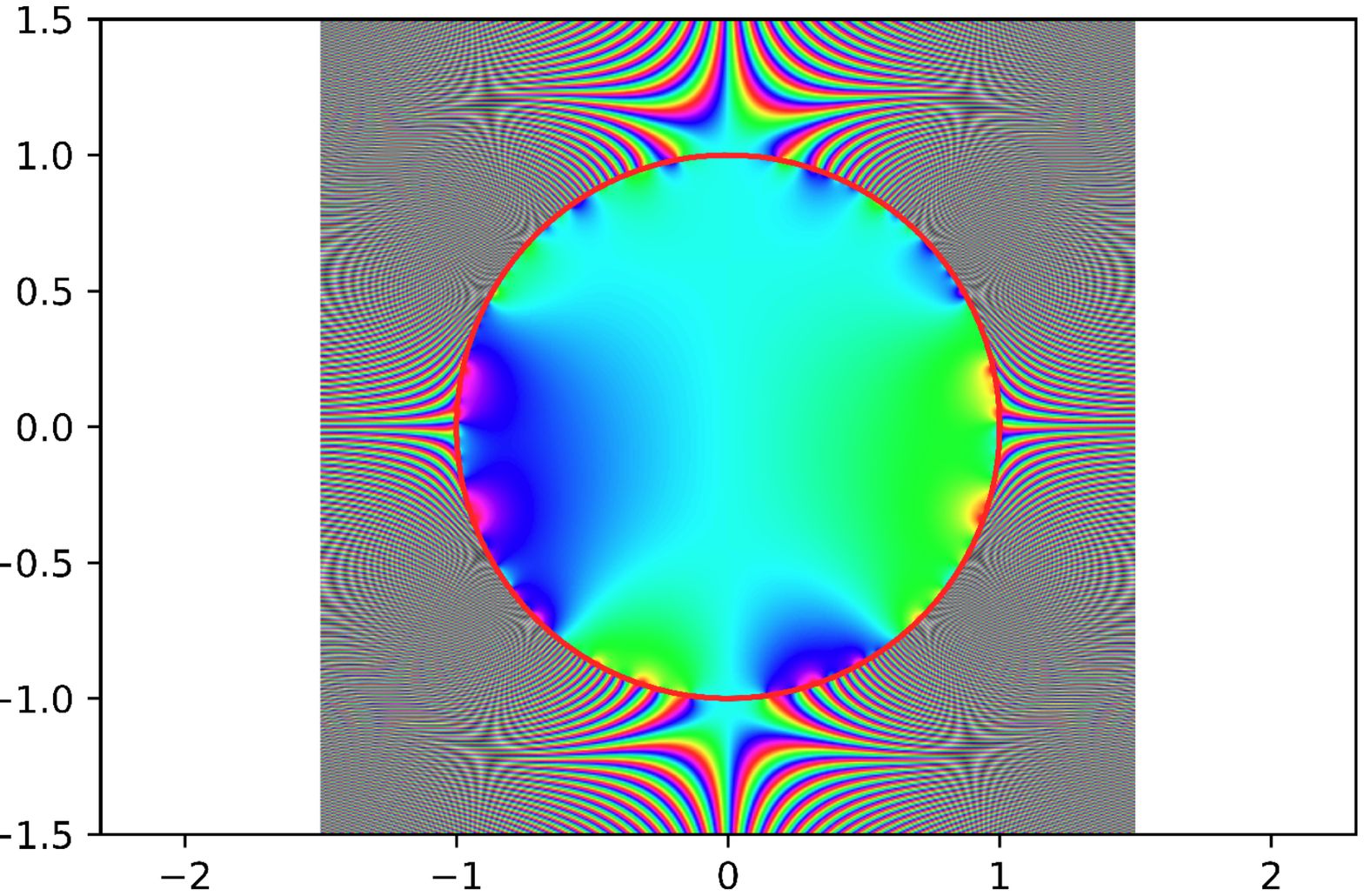} 

\end{minipage} 
\caption{Illustration of Theorem \ref{theorem:main_result}. 
Phase portrait of the normalised characteristic polynomial of an EGE matrix 
of size $250$ for different values of $t$: 0 (top left), $0.3$ (top right), 
$0.6$ (bottom left) and $1$ (bottom right). 
The unit circle is represented in red.}
\label{fig7}
\end{figure}

\noindent
In fact, we expect an analogue of Theorem \ref{theorem:main_result} to hold 
in a much more general setting 
as conjectured in \cite{Bordenave_Chafai_Garcia}.
The limit would only depend on some of the
first four moments of the coefficients of the random matrix
as suggested in Section \ref{subsec_optimal_moments}.
A glimpse
of this universality can be seen, for instance, when calculating
the expected value of the characteristic polynomial.
This depends only on $t = \E[a_{12} a_{21}]$
as explained in
the proof of the following theorem where 
a simple expression for its limit is stated.

\begin{theorem}[Average characteristic polynomial]
  \label{th:cv_mean_poly_univ}
      For each $n$, let 
      $A_{n,t} = (a_{ij}, 1 \leq i,j \leq n)$ be a random matrix such that 
      $\{ (a_{ij}, a_{ji}), 1 \leq i < j \leq n \}$ are i.i.d.\ centered pairs which are
      independent of 
      the i.i.d.\ centered family $\{ a_{ii}, 1 \leq i \leq n \}$ with
      $\E [|a_{ij}|^2] < \infty$  for all $1 \leq i,j \leq n$ and $\E[a_{12}a_{21}] = t \in [0,1]$. 
      Then, for $z$ uniformly in $\D$,
      \begin{equation}
      \label{eq:cv_mean_poly_universal}
          \lim_{n \rightarrow +\infty} \E \bigg[
          \det \left(1 + tz^2 - \frac{z}{\sqrt{n}}A_{n,t} \right) \ed^{-\frac{ntz^2}{2}}  \bigg] = \frac{1}{\sqrt{1 - tz^2}}.
      \end{equation}
\end{theorem}

\subsection{Open questions and comments}

\subsubsection{Extension via matching moments}

Since the way we show tightness is by controlling
the second moment of $f_{n,t}$ and since this second moment 
depends only on the first
four moments of
$A_{n,t}$, tightness of $f_{n,t}$
still holds for the model described in 
Section \ref{subsec_optimal_moments}  
for coefficients $(a_{ij})_{i,j \geq 1}$ whose
first four moments coincide with
those of the EGE. Moreover,
the proof of convergence of the coefficients of $f_{n,t}$
also works in the case where the coefficients
have all moments finite so that 
Theorem \ref{theorem:main_result} holds 
for coefficients $(a_{ij})_{i,j \geq 1}$ with all 
moments finite and whose
first four moments coincide with
those of the EGE.

\subsubsection{Minimal moment condition and universality}
\label{subsec_optimal_moments} 
 As conjectured in \cite{Bordenave_Chafai_Garcia}, 
 the convergence in Theorem \ref{theorem:main_result} of the normalised 
 characteristic polynomial is believed to hold under the minimal moment condition
\begin{equation}
\label{conjecture_opt_moments}
    \E \left[|a_{12} a_{21}|^2 \right] < \infty
\end{equation}
on the entries $(a_{ij})_{i,j \geq 1}$, which gives a condition of a fourth order moment for Wigner 
matrices and second order moment for Girko matrices. The context adapted to
this conjecture is the one of 
elliptic random matrices \cite[Definition 1.3]{Nguyen_ORourke}. 
This model was introduced by Girko in \cite{Girko_elliptic} and \cite{Girko_elliptic_ten} and
a version of this model consists of 
the following random matrices. 
Consider a family  of square-integrable centered 
random variables $(a_{ij})_{i,j \geq 1}$
satisfying that the family of random elements $\{(a_{ij},a_{ji}): i < j\} \cup \{a_{ii}: i \geq 1\}$ 
is independent and whose
law is invariant under any permutation of the indices or, equivalently,
the law of $(a_{ij},a_{ji})$
coincides with the law of $(a_{i'j'}, a_{j'i'})$ whenever $|\{i,j\}| = |\{i',j'\}|$.
If 
\[\E[|a_{12}|^2] = 1 \quad \mbox{ and } \quad \E[a_{12}a_{21}] = t, \]
the matrix $A_n = (a_{ij})_{1 \leq i,j \leq n}$ is said to be
$t$-Girko.
The convergence of the average eigenvalue distribution towards
the uniform distribution on the ellipse 
\eqref{eq:ellipse} has been proved
under different conditions on the variables, see
\cite{Nguyen_ORourke,ORourke_Renfrew,Naumov}. We expect the following version of
Theorem \ref{theorem:main_result} to hold for the general
$t$-Girko matrices
described above. Denoting 
$\tau = \E[a_{12}^2]$, $s = \E[a_{11}^2] - t - \tau$ and 
$q = \E [(a_{12}a_{21}-t)^2] - t^2 - \tau^2$, the limit of 
$\det\big(1 + tz^2 - z \frac{A_{n,t}}{\sqrt{n}} \big) \exp(-ntz^2/2)$ 
is expected to be given by
\begin{equation}
    \label{eq:universal_limit}
    \sqrt{1 - \tau z^2} \ed^{-sz^2/2} \ed^{-qz^4/4} 
    \ed^{- \sum_{k \geq 1} Y_k \frac{z^k}{\sqrt{k}}}
\end{equation}
where $(Y_k)_{k \geq 1}$ are independent centered complex Gaussians 
such that 
$Y_1$ has the same variance as $a_{11}$,
$Y_2$ has the same variance as
$a_{12} a_{21}$  and, for $k \geq 3$, 
the variance of $Y_k$ is the sum of
the $k$-th power of the 
variance of $a_{12}$ and the $k$-th power of the covariance of 
$a_{12}$ and $a_{21}$ or, more explicitly,
$\E[Y_k^2] = \E[a_{12}^2]^k + \E[a_{12}a_{21}]^k = \tau^k + t^k$ and 
$\E[|Y_k|^2] = \E[|a_{12}|^2]^k + \E[a_{12} \overline{a_{21}}]^k
=1 + \E[a_{12} \overline{a_{21}}]^k$.

\subsubsection{Matrices with entries in  \texorpdfstring{$\{ 0, 1 \}$}{Lg}}
As described above, a convergence of the reciprocal 
characteristic polynomial for 
matrices with independent Bernoulli entries with non-zero 
expectation has been proved 
in \cite{Coste}. 
The limiting random holomorphic function can be expressed using 
Poisson random variables, see \cite[Theorem 2.3]{Coste}. 
One could ask for an analogue of the Elliptic Ginibre Ensemble 
for such matrices and for the convergence 
of its normalised characteristic polynomial.
\\
\\
Extension of the work \cite{Coste_Lambert_Zhu_permutations} for 
Ewens distributed random permutations can be found in \cite{Francois_Ewens}.
The convergence of 
traces of such random matrices were understood from the work of 
Nikeghbali and Zeindler \cite{Nikeghbali_Zeindler}, while tightness uses asymptotics 
of Hwang \cite{Hwang}.

\subsubsection{Determinantal Coulomb gases}

As explained in \ref{subsec_optimal_moments}, 
this work can be thought of as a first step towards the convergence
of the characteristic polynomial outside the support of the equilibrium measure for general elliptic 
random matrices. Nevertheless, we could have followed a different path, which is to look the Elliptic 
Ginibre Ensembles as a particular case of a determinantal Coulomb gas. In this
vein, it may be possible to show the convergence of the traces by adapting results from
\cite{Ward}
and to show tightness of the characteristic polynomial outside the support of the equilibrium measure for 
more general determinantal Coulomb gases by using, for instance, 
the results from
\cite{SzegoAsymptotics}.

\section{Proof of Theorem  \ref{theorem:main_result}}

From now on, given that the case $t=0$ is already
treated in \cite{Bordenave_Chafai_Garcia}, we assume $t \neq 0$ and
 omit the index $t$ since it is considered fixed for the rest of the article.\\
\\
As is standard, to show that $\{f_{n}\}_{n \geq 1}$ converges we show that
$\{f_{n}\}_{n \geq 1}$ is tight and 
that the coefficients in its power-series expansion around the origin
converge in law. We state this classical fact as Lemma \ref{lemma:conditions_a_b} 
below
and we refer to \cite[Section 4.2]{Bordenave_Chafai_Garcia} for a proof.
Recall that $\mathcal{H}(\D)$ is the space of holomorphic functions on the unit disk 
$\D$ endowed with the topology of uniform convergence 
on compact sets. 

\begin{lemma}[Tightness and convergence of coefficients imply convergence of functions]
	\label{lemma:conditions_a_b}
	Let $\{h_n \}_{n \geq 1}$ be a sequence of random elements in $\mathcal{H}(\D)$ and denote 
	the coefficients of $h_n$ by $(\xi_k^{(n)} )_{k \geq 0}$ so that for all $z \in \D$, 
	$h_n(z) = \sum_{k \geq 0} \xi_k^{(n)} z^k$. Suppose also that the following conditions hold.
	\begin{itemize}
		\item[$(a)$] The sequence $\{h_n \}_{n \geq 1}$ is
		a tight sequence of random elements of $\mathcal{H}(\D)$.
		\item[$(b)$] 
		There exists a sequence $(\xi_k)_{k \geq 0}$ of 
		random variables such that,
		for every $m \geq 0$, the vector $(\xi_0^{(n)}, \dots, \xi_m^{(n)})$
		converges in law as $n \to \infty$ to $(\xi_0, \dots, \xi_m)$.
	\end{itemize}
	Then, $h(z) = \sum_{k \geq 0} \xi_k z^k$ is a well-defined function in $\mathcal{H}(\D)$ and $h_n$
	converges in law towards $h$ in $\mathcal{H}(\D)$ for the topology of local uniform convergence.
\end{lemma}

\noindent
The first step is to show the following theorem.

\begin{theorem}[Tightness]
	\label{theorem:tightness}
	The sequence $\{ f_{n} \}_{n \geq 1}$ is tight.
\end{theorem}

\noindent
The proof uses known properties of the Elliptic Ginibre Ensemble 
and its relation to scaled 
Hermite polynomials. In particular, it relies on the 
determinantal aspect of its eigenvalue point process.
A local uniform control is derived from  \cite{Akemann_Duits_Molag}. 
This control allows us to derive tightness thanks to 
Montel's theorem as stated in Lemma \ref{lemma:reduction_to_unif_control} below. 
This ensures that Condition (a) of Lemma \ref{lemma:conditions_a_b}
is satisfied for $\{f_n\}_{n \geq 1}$.
\\
\\
To guarantee Condition (b) of Lemma \ref{lemma:conditions_a_b}, we
 express the coefficients of $\{f_n\}_{n \geq 1}$ using a family of polynomials that 
we call the modified Chebyshev polynomials.

\begin{definition}[Modified Chebyshev polynomials]
	\label{def:modif_cheby_poly}
	The \textit{modified Chebyshev polynomials} are the polynomials 
	$\big\{ P_k \big\}_{k \geq 1}$ satisfying the recurrence relation 
	\begin{equation}
		\label{eq:recurrence_modif_cheby}
		P_{k+1} = X P_k - tP_{k-1}, \ \ P_1 = X, \ \  P_2 = X^2 - 2t.
	\end{equation}
\end{definition}

\noindent
We may also give the modified Chebyshev polynomials more explicitly
	by their coefficients
	\begin{equation}
		\label{eq:expression_coefs_alpha}
		\alpha^{(k)}_{k-2j} \coloneqq (-t)^j \frac{k}{k-j} \binom{k-j}{j}
	\end{equation}
	so that $P_k = \sum_{j \geq 0} \alpha^{(k)}_{k-2j} X^{k-2j}$ or
	by its generating function
	\begin{equation} 
		\label{eq:generating}
		\sum_{k \geq 1} P_k(w) \frac{z^k}{k} = - \log (1 + tz^2 - zw).
	\end{equation}
	The latter expression can be obtained either by
	noticing that $P_k(w)= 2\sqrt t^k T_k(w/(2\sqrt t))$ 
	for $T_k$ the classical Chebyshev polynomials
	of the first kind or by taking the derivative in $z$ of
	\eqref{eq:generating}, which is 
	$(w-2z t)(1+tz^2 - zw)^{-1}$, so that
	$(\sum_{k \geq 0} P_{k+1}(w)z^k)(1+tz^2 - zw) = w-2zt$ which is
	the recurrence 	\eqref{eq:recurrence_modif_cheby}.\\
	\\
	Then, for $z \in \D$ small enough using \eqref{eq:generating}
	\begin{equation*}
		1 + tz^2 - \left(\frac{A_{n}}{\sqrt{n}} \right)z  =  \exp 
		\left( - \sum_{k \geq 1} P_k \left(\frac{A_{n}}{\sqrt{n}} \right) \frac{z^k}{k} \right),
	\end{equation*}
	by, for instance, showing this for diagonalizable matrices
	and the using density of this set of matrices. Finally, taking the determinant
	we obtain that, for $z \in \D$ small enough,
	\begin{equation}
		\label{trace_expression}
		f_{n}(z) := \det \Big(1 + tz^2 - \frac{z}{\sqrt{n}}A_{n}\Big)  \ed^{- \frac{ntz^2}{2}}
		\,  = \, \exp \left( -\sum_{k \geq 1} U_{k}^{(n)} \frac{z^k}{k}\right),
	\end{equation}
	where 
	\begin{equation}
		\label{eq:Uk}
		U_{k}^{(n)} \coloneqq 
		\mathrm{Tr} \left[ P_k \left(\frac{A_{n}}{\sqrt{n}} \right) \right] + nt \delta_{k=2}.
	\end{equation}
	In particular, the first
	$m$ coefficients of 
	$f_{n}$  can be expressed as polynomials which are independent of $n$
	of $U_{0}^{(n)}, \dots, U_{m}^{(n)} $ and vice versa.
	Thus, 
	showing the convergence in law of the variables $(U_{k}^{(n)})_{k \geq 1}$
	is equivalent to showing the convergence in law of 
	the coefficients of $f_{n}$, which is our way of characterizing the
	limit in Condition (b) of
	Lemma \ref{lemma:conditions_a_b}. 
	Since it is easier to deal with traces, we will study $U_{k}^{(n)}$ 
	and prove the convergence stated in the following theorem.

\begin{theorem}[Convergence of the traces of Chebyshev polynomials]
	\label{proposition:convergence_to_gaussian}
	
	\[\big(U_{k}^{(n)}\big)_{k \geq 1} 
	\xrightarrow[n \to \infty]{\phantom{aa} \mathrm{law} \phantom{aa}}
	(\sqrt k X_{k})_{k \geq 1},
	\]
	where $(X_{k})_{k \geq 1}$ is a family of independent centered complex Gaussian random variables such 
	that $\E[X_{k}^2] = t^k$ and $\E[|X_{k}|^2] = 1$. 
\end{theorem}

\begin{proof}[Conclusion of the proof of Theorem \ref{theorem:main_result}]
	We use Lemma \ref{lemma:conditions_a_b} with $h_n = f_n$
	and $\xi_k$ being the coefficients of
	$h(z)=e^{-F_t(z)}=\exp(-\sum_{k\geq 1} X_k \frac{z^k}{\sqrt k})$.
	Theorem	\ref{theorem:tightness} 
	ensures that Condition (a) 
	 is satisfied while
	Theorem	\ref{proposition:convergence_to_gaussian} 
	ensures that Condition (b) 
	is satisfied.
\end{proof}

\subsection{Tightness: Proof of Theorem \ref{theorem:tightness}}
\label{section_tightness}

\noindent
Recall that $f_{n}:\D \to \C$ is given by
\[   
    f_{n}(z) = \det \Big(1 + tz^2 - \frac{z}{\sqrt{n}}A_{n}\Big)  \ed^{- \frac{ntz^2}{2}}.
\]
\noindent
and the space 
   $\mathcal{H}(\D)$ of holomorphic functions on $\D$ is endowed with the 
   topology of uniform convergence on compact sets. Lemma \ref{lemma:reduction_to_unif_control} 
   below is the stochastic version of Montel's theorem which reduces the proof of tightness to a control on compact sets.

\begin{lemma}[Montel's theorem]
\label{lemma:reduction_to_unif_control}
    Suppose that for every compact $K \subset \D$, the sequence 
    $(||f_{n}||_K)_{n \geq 1}$ is tight, where $||f_{n}||_K = \max_{z \in K} |f_{n}(z)|$. 
    Then, $\{f_{n} \}_{n \geq 1}$ is tight.
\end{lemma}

\begin{proof}
    It is a consequence of the classical Montel's theorem of complex analysis.
    See, for instance, \cite[Proposition 2.5]{Shirai}.
\end{proof}

\begin{remark}
\label{remark:from_max_to_exp} 
By the subharmonicity of $|f_{n}(z)|^2$, 
saying that $(\mathbb E[\|f_{n}\|_K^2])_{n \geq 1}$ 
is a bounded sequence for every compact $K \subset \mathbb D$ 
is equivalent to saying that
$(\sup_{z \in K}\E[|f_{n}(z)|^2])_{n \geq 1}$ 
is a bounded sequence for every compact $K \subset \mathbb D$.
    See, for instance, \cite[Lemma 2.6]{Shirai}. In the Girko case of \cite{Bordenave_Chafai_Garcia}, one had a 
    remarkable orthogonality of the sub-determinants which led to an upper bound on the 
    desired quantity. As we no longer have this property, our proof is based on \cite{Akemann_Vernizzi} which exploits the integrability of the Elliptic Ginibre Ensemble.
\end{remark}
\vspace{2mm}
\noindent
Our goal is to control
$\E [| f_{n}(z) |^2]$ and, for this, we will
begin by giving an explicit expression of the second moment using
Hermite polynomials. 

\begin{definition}[Hermite polynomials]
    The \textit{Hermite polynomials} $\{ H_n \}_{n \geq 0}$ are the monic orthogonal polynomials 
    with respect to the measure $\ed^{-x^2/2} \ddiff x$ on $\R$ so that 
    \begin{equation*}
    \int_{\R} H_n(x) H_m(x) \ed^{-\frac{x^2}{2}}\ddiff x = \sqrt{2\pi} n! \delta_{n,m}.
    \end{equation*}
\end{definition}

\noindent
Recall that $g_t(z) = z^{-1} + tz$. 

\begin{lemma}[Hermite expression of the characteristic polynomial]
\label{lemma:hermite_expression}
For $n \geq 1$ and $z \in \D \setminus \{0\}$, one has the following expression
    \begin{equation}
    \label{eq:hermite_expression}
    \E [| f_{n}(z) |^2] = \frac{n!|z|^{2n}}{n^n}  \left|\ed^{-ntz^2}\right| 
    \sum_{k=0}^n \frac{t^k}{k!} \left| H_{k}\left(\sqrt{\frac{n}{t}}g_t(z) \right) \right|^2.
\end{equation}
\end{lemma}

\begin{proof}
	In the case of the Elliptic Ginibre Ensemble given by 
	\eqref{eq:elliptic_sum_gue}, the matrix $A_{n,t}$ has
	the following density, which can be found in \cite[eq. (4)]{Akemann_Duits_Molag}.
	
	\begin{equation}
		\label{law_matrix_ellipse}
		\ddiff \proba_t(M) = \left( \frac{1}{\pi \sqrt{1-t^2}} \right)^{n^2} 
		\exp \left( - \frac{1}{1-t^2} \mathrm{Tr} \left[MM^* - \frac{t}{2}(M^2 + (M^*)^2) \right] \right)\ddiff M
	\end{equation}
	which has the form $\ddiff \proba_t(M) = w_t(M)\ddiff M$ associated to the weight function 
	\begin{eqnarray}
		w_t(z)             & = & \frac{1}{\pi \sqrt{1 - t^2}} 
		\exp{ \left( - \frac{1}{1-t^2} \left(|z|^2 - \frac{t}{2}(z^2 + \overline{z}^2) \right) \right) }  \nonumber\\
		& = & \frac{1}{\pi \sqrt{1 - t^2}} 
		\exp{ \left( - \left( \frac{x^2}{1+t} + \frac{y^2}{1-t} \right) \right)} \nonumber
	\end{eqnarray}
	with $x = \mathrm{Re}(z)$ and $y= \mathrm{Im}(z)$. 
	In order to use the main theorem of \cite{Akemann_Vernizzi}, we should compute the orthonormal 
	polynomials with respect to $w_t(z) \ddiff z$. Using \cite[eq. (3)]{Akemann_Duits_Molag}, 
	these polynomials are $\{ P_n \}_{n \geq 0}$ given by 
	
	\begin{equation}
		P_n(z) = \frac{\sqrt{t^n}}{\sqrt{n!}} H_n \left( \frac{z}{\sqrt{t}} \right).
	\end{equation}
	 For $M$ sampled from \eqref{law_matrix_ellipse}
	and any $u,v \in \C$, we may use \cite[eq. (2.11)]{Akemann_Vernizzi} to get
	
	\begin{equation}
		\E \left[ \det(u-M) \overline{\det(v-M)} \right] = n! \sum_{k=0}^n P_k(u) \overline{P_k(v)},
	\end{equation}
	where the global factor $n!$ is the square inverse of the dominant coefficient
	of $P_n$.
	\noindent
	Finally, setting $u=v = g_t(z) = z^{-1}+tz$ gives
	\begin{align*}
		\E [| f_{n}(z) |^2] &= 
		\E \bigg[\bigg| \ed^{-\frac{ntz^2}{2}} \left( \frac{z}{\sqrt{n}} \right)^n 
		\det \left( \sqrt{n}(z^{-1}+tz) - A_{n,t} \right) \bigg|^2 \bigg]											\\
		&=\frac{n! |z|^{2n}}{n^n} \left| 
		\ed^{-ntz^2} \right| \sum_{k=0}^n \frac{t^k}{k!} \left| H_k \left( \sqrt{\frac{n}{t}}g_t(z)\right) \right|^2
	\end{align*}
	which is the desired expression of $\E [| f_{n,t}(z) |^2]$ in terms of Hermite polynomials.
	
\end{proof}

\noindent
With the help of the expression \eqref{eq:hermite_expression} and using 
the results from \cite{Akemann_Duits_Molag}, we will control 
$\E [| f_{n}(z) |^2]$ uniformly on bounded sets. In fact,
\cite{Akemann_Duits_Molag} allows us to give an 
explicit expression for the limit of
$\E [| f_{n}(z) |^2]$. Since we do not need 
an explicit expression, we will only state the following.

\begin{lemma}[Convergence of the second moment]
\label{lemma:asymptotics_second_moment}
 There exists 
 a strictly positive continuous function
 $\mathcal F: \mathbb D \setminus \{0\} \to (0,\infty)$ 
 such that, uniformly on compact sets,
\[\E [| f_{n}|^2] \xrightarrow[n \to \infty]{} \mathcal F.\]
\end{lemma}

\noindent
Since $f_{n}$ is holomorphic on
$\mathbb D$, the function 
$| f_{n}|^2$ on the circle of radius $r \in (0,1)$ controls $|f_{n}|^2$
on the disk $\D_r = \{ z \in \C: |z| < r \}$ for 
$r \in (0,1)$. This is written in the next proposition.

\begin{proposition}[Uniform control]
\label{proposition:uniform_control}
   	For every $r \in (0,1)$ there exists
    $C_r > 0$ such that
    \[\E[\|f_{n}\|_{\D_r}^2]\leq C_r \mbox{ for every } n \geq 1.\]
\end{proposition}
\begin{proof}
	By Lemma \ref{lemma:asymptotics_second_moment},
	we have a bound for
	$\mathbb E [|f_{n}(z)|^2]$
	on compact sets of $\mathbb D\setminus \{0\}$.
	This is the same as a bound for
	$\mathbb E [\|f_{n} \|_K^2]$
	for compact sets $K \subset \mathbb D \setminus \{0\}$ by
	Remark \ref{remark:from_max_to_exp}. 
	We may obtain a bound for
	$\mathbb E[\|f_{n}\|_{\D_r}^2]$ for $r \in (0,1)$ by using that
	$\|f_{n}\|_{\D_r} \leq \|f_{n}\|_{\partial \D_r}$
	thanks to the maximum modulus principle.
\end{proof}

\begin{proof}[Proof of Lemma \ref{lemma:asymptotics_second_moment}]
	
	Recall the function
	$g_t: \mathbb D \to \mathbb C\setminus \mathcal E_t$
	given by $g_t(z) =\frac{1}{z}+tz $
	and define $L_n:\mathbb C \setminus \mathcal E_t \to [0,\infty)$ by
	
	\begin{equation*}
		L_n(u)=\sum_{k=0}^{n-1} \frac{t^k}{k!}
		\left|H_k \left(\sqrt {\frac{n}{t}} u \right)\right|^2.
	\end{equation*}
	By using the contour integral representation around a small
	loop enclosing the origin, 
	\[L_n(u)=
	\frac{1}{2\pi i}
	\oint \displaylimits_0 \frac{\ed^{nF_u(s)}}{t - s} \frac{\mathrm d s}{\sqrt{1-s^2}},
	\quad \mbox{ with } \quad
	F_u(s) = \frac{s}{t} \left(\frac{\mathrm{Re}(u)^2}{1+s}
	+ \frac{\mathrm{Im}(u)^2}{1-s} \right)
	- \log s + \log t,\]
	the following has been proved in 
	\cite[Theorem II.12, (i)]{Akemann_Duits_Molag}
	and \cite[Theorem II.13, (i)]{Akemann_Duits_Molag} for
	$u \in \mathbb C \setminus \mathcal E_t$ and $z=g_t^{-1}(u)$,
	\begin{equation}
		\label{eq:Asymptotics}
		L_n(u)
		= \frac{1}{2\pi }
		\sqrt{\frac{2\pi}{n F_u''(t |z|^2)}}
		\frac{\ed^{n F_u(t|z|^2)}}{\sqrt{1-t^2 |z|^4}}
		\frac{1}{ t(1 - |z|^2)}
		\left(1 + O\left(\frac{1}{n} \right) \right),
	\end{equation}
	where the error term is uniform on compact sets
	of $\mathbb C \setminus \mathcal E_t$.
	Here $s=t|z|^2$ is a critical point of $F_u$ which can be seen by first
	calculating
	\begin{align*}
		tF_u'(s) &= \frac{\mathrm{Re}(u)^2}{1+s} +  \frac{\mathrm{Im}(u)^2}{1-s}
	+ s \Big(\hspace{-1mm} - \frac{\mathrm{Re}(u)^2}{(1+s)^2} + \frac{\mathrm{Im}(u)^2}{(1-s)^2}
	\Big)
		- \frac{t}{s}																									\\
		&=
		\frac{\mathrm{Re}(u)^2}{(1+s)^2} +  \frac{\mathrm{Im}(u)^2}{(1-s)^2}
		-\frac{t}{s}																										
		\\
		&=		\frac{t}{s}\Bigg[
		\frac{\mathrm{Re}(u)^2}{\Big(\frac{1}{\sqrt{s/t}}+t \sqrt{s/t}\Big)^2} +  
		\frac{\mathrm{Im}(u)^2}{\Big(\frac{1}{\sqrt{s/t}}-t \sqrt{s/t}\Big)^2}
		-1
		\Bigg]
	\end{align*}
	and then noticing  that $g_t$ sends $\{z \in \mathbb C: |z| = r\}$
	to 
		$\big\{u \in \mathbb C: \frac{\mathrm{Re}(u)^2}{(\frac{1}{r} + tr)^2} + 
		\frac{\mathrm{Im}(u)^2}{(\frac{1}{r} - tr)^2}= 1 \big\} \vspace{-0.7mm} $. 
	Moreover,
	we can find
		\begin{align*}
		F_u(t |z|^2) &=
		|z|^2 \left(
		\frac{\mathrm{Re}\left(u \right)^2}{1+t|z|^2}
		+
		\frac{\mathrm{Im} \left(u \right)^2}{1-t|z|^2}
		\right)
		- \log(t |z|^2) + \log t												\\	
		&=1+ t|z|^4 \left(
		\frac{\mathrm{Re}\left(u\right)^2}{(1+t|z|^2)^2}
		-
		\frac{\mathrm{Im} \left(u \right)^2}{(1-t|z|^2)^2}
		\right)
		- \log(|z|^2) 																	\\
		&=1 + t \mathrm{Re}(z^2) - \log(|z|^2),									
	\end{align*}
	where for the second equality we have used
	that $F_u'(t|z|^2) = 0$ to simplify the calculation
	and for the third equality we have used  
	that $u = z^{-1} + tz$ allows us to relate the real parts
	$(t|z|^2 + 1)\mathrm{Re}(z) = |z|^2\mathrm{Re}(u)$
	and the imaginary parts
	$(t|z|^2 - 1) \mathrm{Im}(z) = |z|^2\mathrm{Im}(u)$.
	In our case
	we need to control the second moment
	\begin{align*}
		\mathbb E [|f_{n}(z)|^2] &= 
		\frac{n!|z|^{2n}}{n^n} \ed^{-nt \mathrm{Re}(z^2)} \sum_{k=0}^n \frac{t^k}{k!} \left| H_k
		\left(\sqrt{\frac{n}{t}}\left( \frac{1}{z}+tz \right) \right) \right|^2 			\\
		&=
		\frac{n!|z|^{2n}}{n^n} \ed^{-nt \mathrm{Re}(z^2)} L_{n+1}
		\left(\sqrt \frac{n}{n+1} g_t(z)
		\right).
	\end{align*}
	By \eqref{eq:Asymptotics} 
	and Stirling's formula, we immediately notice that
	\begin{equation*}
		\frac{n!|z|^{2n}}{n^n} \ed^{-nt \mathrm{Re}(z^2)} L_{n}
		(g_t(z)) \xrightarrow[n \to \infty]{}
	\frac{1}{ \sqrt{ F_u''(t |z|^2)(1-t^2 |z|^4)}
		t(1 - |z|^2)}.
		\end{equation*}
	uniformly on compact sets of $\mathbb D\setminus \{0\}$.
	It is now enough to notice that the quotient 
    $$L_{n+1} (\sqrt{n/(n+1)}g_t(z))/L_{n}(g_t(z))$$
	converges uniformly on compact sets
	towards a nowhere zero function. 
	To this end, we must compute the limit of  
	$\exp((n+1) G(\sqrt{n/(n+1)}u)
	- n G(u)),$
	where $G(w) =F_w(t|g_t^{-1}(w)|^2) $. But,
	for $u$ uniformly on compact sets 
	of $\mathbb C \setminus \mathcal E_t$,
	\[
	\ed^{(n+1) G(\sqrt{\frac{n}{n+1}} u)
		- n G(u)} \xrightarrow[n \to \infty]{} \ed^{G(u) - \frac{1}{2}\langle \nabla G(u), u\rangle}\]
	so that the proof is complete 
	with $\mathcal{F}(z) = 
	\frac{\ed^{G(u) - \frac{1}{2}\langle \nabla G(u), u\rangle}}{ \sqrt{ F_u''(t |z|^2)(1-t^2 |z|^4)}
		t(1 - |z|^2)}$, $u = g_t(z)$.
\end{proof}

\begin{proof}[Conclusion of the proof of Theorem \ref{theorem:tightness}]

For every compact subset $K \subset \D$, Proposition \ref{proposition:uniform_control} 
implies that $(||f_n||_K)_{n \geq 1}$ has its second moment uniformly bounded. 
The sequence is therefore tight from which one derives the tightness of $\{f_n\}_{n \geq 1}$ 
using Lemma \ref{lemma:reduction_to_unif_control} and Remark \ref{remark:from_max_to_exp}.

\end{proof}
\noindent

\subsection{Convergence of the coefficients: Proof of Theorem \ref{proposition:convergence_to_gaussian}}
\label{section_convergence_coefs}

We will divide the proof in three parts. First, we show that the expected value
$E\big[ U_k^{(n)} \big]$ converge to zero. 
As a second step we show that the fluctuations are Gaussian.
Finally, we identify the covariance.
We will deal with traces of polynomials of $A_n$. Recall that 
 \begin{equation*}
	\mathrm{Tr}[A_n^k] = \hspace{-2mm} \sum_{(i_1, \dots, i_k) \in \{1,\dots,n\}^k}
	\hspace{-2mm} a_{i_1  i_2} a_{i_2  i_3} \dots a_{i_{k-1}  i_k} a_{i_k, i_1}.
\end{equation*}
It is convenient
to think of $(i_1,\dots,i_k)$ as a path of length $k$ with values in $\{1,\dots,n\}$.
We use the notation $[m] = \{1,\dots,m\}$ which will be thought
of as the abelian group $\mathbb Z/m\mathbb Z$ when performing additions,
and for $\psi:[k] \to [n]$ we denote
$a_\psi = \prod_{i=1}^k a_{\psi(i) \psi(i+1)}$.
Due to the invariance under permutations of the law of 
the entries in the matrix $A_n$,
quantities such as
$\mathbb E[a_\psi]$ or $\mathbb E[a_\psi a_\varphi]$
are invariant under $\psi \mapsto T\circ \psi$ for any permutation
$T: [n] \to [n]$, i.e., 
$\mathbb E[a_\psi] = 
\mathbb E[a_{T \circ \psi}]$ or $
\mathbb E[a_\psi a_\varphi]=\mathbb E[a_{T \circ \psi} a_{T \circ \varphi}]$.
In the following, we may identify
$\psi$ and $T \circ \psi$ or $(\psi,\varphi)$ and $(T \circ \psi, T \circ \varphi)$
for some purposes. \\
\\
One can describe these
equivalence classes by considering the 
partition induced by $\psi$ or by $(\psi,\varphi)$ as follows.
Let $k_1,\dots,k_\ell$ be positive integers and let us use the notation
$\mathbf k = (k_1,\dots,k_\ell)$.
First recall that the information of
$\ell$ maps $(\psi_j:[k_j] \to [n])_{1 \leq j\leq \ell}$
is contained in the map 
$\psi = \sqcup \psi_j: \sqcup_{j=1}^\ell [k_j] \to [n]$.
For simplicity, we denote $[\mathbf k] = \sqcup_{j=1}^\ell [k_j]$. \\

\begin{definition}[Path pattern]
Let $\psi : [\mathbf k] \to [n]$.
	The \textit{path pattern} traced by $\psi$,
	or the \textit{partition} induced by $\psi$
	is the partition of $[k]$ given by  
	\begin{equation*}
		\pi_\psi \coloneqq \{\psi^{-1}(x): x \in 
	\psi([\mathbf k]) \}.
	\end{equation*} 
\end{definition}

\noindent
We denote by $\mathcal P(\mathbf k)$ 
the set of path patterns, or partitions, of $[\mathbf k]$.
We use these path patterns to write, for instance,
 \begin{equation} \label{eq:Api}
	\mathrm{Tr}[A_n^k] = \sum_{\pi \in \mathcal P(k)} 
	\mathcal A_\pi^{(n)}, \quad \mbox{ where } 
	\mathcal A_\pi^{(n)} \coloneqq \sum_{\substack{\psi:[k] \to [n] \vspace{0.5mm}  \\ \pi_\psi = \pi}} a_\psi.
\end{equation}
Since the set $\mathcal P(k)$ of possible path patterns 
is finite, it will be convenient to study
the limit of $\mathcal A_\pi^{(n)}$ for a given $\pi \in \mathcal P(\mathbf k)$. 
We chose to consider
path patterns because they are convenient to count.
We proceed to define
the directed multigraph $G_\psi$
associated to a map $\psi = \sqcup_{j=1}^\ell \psi_j :[\mathbf k] \to [n]$
in a more explicit way in Definition \ref{def:graph_psi}. 
The motivation is that 
$\mathbb E[a_\psi]$ depends only on isomorphism class of
the graph $G_\psi$. \\

\begin{definition}[Graph of $\psi$]
	\label{def:graph_psi}
	Let $\psi = \sqcup_{j=1}^\ell \psi_j :[\mathbf k] \to [n]$. 
	The graph $G_\psi = (V,E,s,t)$ is the directed multigraph having
	vertex set $V = \psi([\mathbf k])$,
	edge set $E = \sqcup_{j=1}^\ell E_j$ where we defined
	$E_j = \{(i,i+1): i \in [k_j]\}$,
	source map $s = \sqcup_{j=1}^\ell s_j$
	with $s_j:E_j \to V$ given by $s(i,i+1) = \psi_j(i)$
	and target map
	$t = \sqcup_{j=1}^\ell t_j$
	with $t_j:E_j \to V$ given by $t(i,i+1) = \psi_j(i+1)$. \\
\end{definition}

\noindent
More concisely but equivalently for our purposes,
if $I_j: [k_j] \to [\mathbf k]$ are the canonical inclusion maps,
we may say that $G_\psi$ is the directed multigraph with vertex set
$V = \psi([\mathbf k])$ and  edge multiset
$\{(I_j(i),I_j(i+1)): j \in [\ell], i \in [k_j] \}$.
For a directed multigraph $G = (V,E,s,t)$, denote by $\overline{G} = (V, \overline{E})$ 
the undirected simple graph with edge set $\overline E = \{\{s(e),t(e)\} :
e \in E \}$. 
So, $\overline G_\psi$ is the undirected simple graph
(with possible loops)
induced by $\psi$.\\
\\
Notice that the isomorphism class of $G_\psi$ depends only on
the path pattern $\pi \in \mathcal P(\mathbf k)$ induced by $\psi$.
We denote by $G_\pi$ this isomorphism class.
When dealing with $G_\pi$ 
for purposes where the quantities only depend on the 
path pattern, we may write $V_\pi$
and $E_\pi$ for its set of vertices and edges viewed as
$V_\psi$ and $E_\psi$ for some $\psi$ inducing the partition $\pi$.
The same goes for $\overline G_\pi$.
Moreover, if $\psi = \sqcup_{i=1}^\ell \psi_j$ and $\pi = \pi_\psi$,
we can consider $\pi_j$, the path pattern induced by $\psi_j$
or, what is the same, the partition induced
by $\pi$ via the canonical inclusion $I_j:[k_j] \to [\mathbf k]$.
It will then be convenient to see $G_{\pi_i}$ as a subgraph of $G_\pi$
and $\overline G_{\pi_i}$ as a subgraph of $\overline G_\pi$.
This suggest to consider
the equivalence class of $(G_\psi,G_{\psi_1},\dots,G_{\psi_\ell})$
and viewing $G_{\pi_i}$ as a subgraph of $G_\pi$ by taking 
representatives of this equivalence class.\\
\\
We will see below in the proof of 
Lemma \ref{lemma:expectation_monomial} and in Proposition 
\ref{prop:GaussianLimit} the following
behaviors for
$n^{-k/2}\mathcal A_\pi^{(n)}$.
The first two cases correspond to the random part in the 
limit whereas the last two cases give deterministic contributions.

\begin{enumerate}
	\item If $\overline G_\pi$ is unicyclic, i.e., if 
	it contains one and only one cycle (that may be a loop), with each edge of the cycle
	being traversed just once 
	and every edge outside the cycle being traversed
	once in each direction in $G_\pi$, then
	$n^{-k/2}\mathcal A_\pi^{(n)}$ converges to a centered Gaussian. 
	\item If $\overline G_\pi$ is a tree where each edge is 
	traversed twice in $G_\pi$, once in each direction
	($\pi$ is called a rooted plane tree in this case, 
	see Definition \ref{def:rooted_plane_tree_and_unicyclic}), then
	$n^{-k/2}\mathcal A_\pi^{(n)} - n $ converges to a non-centered Gaussian.
	This case can be thought of as a unicyclic graph with a cycle of length two.
	\item If $\overline G_\pi$ is a tree but there is an edge traversed
	four times in $G_\pi$, two in each direction, then 
	$n^{-k/2}\mathcal A_\pi^{(n)}$ converges to a constant.
	Those $\pi$ can be thought of as rooted plane trees
	with two vertices at distance two identified.
	\item If $\overline G_\pi$ is unicyclic but with each 
	edge traversed two times, once in each direction,
	then $n^{-k/2}\mathcal A_\pi^{(n)}$ also converges to a constant.
	Those $\pi$ can be thought of as rooted plane trees
	with two  vertices at distance different from two identified.
	\item All other cases converge to zero.
\end{enumerate}

\begin{figure}[ht] 
		\centering
		\includegraphics[scale=0.5]{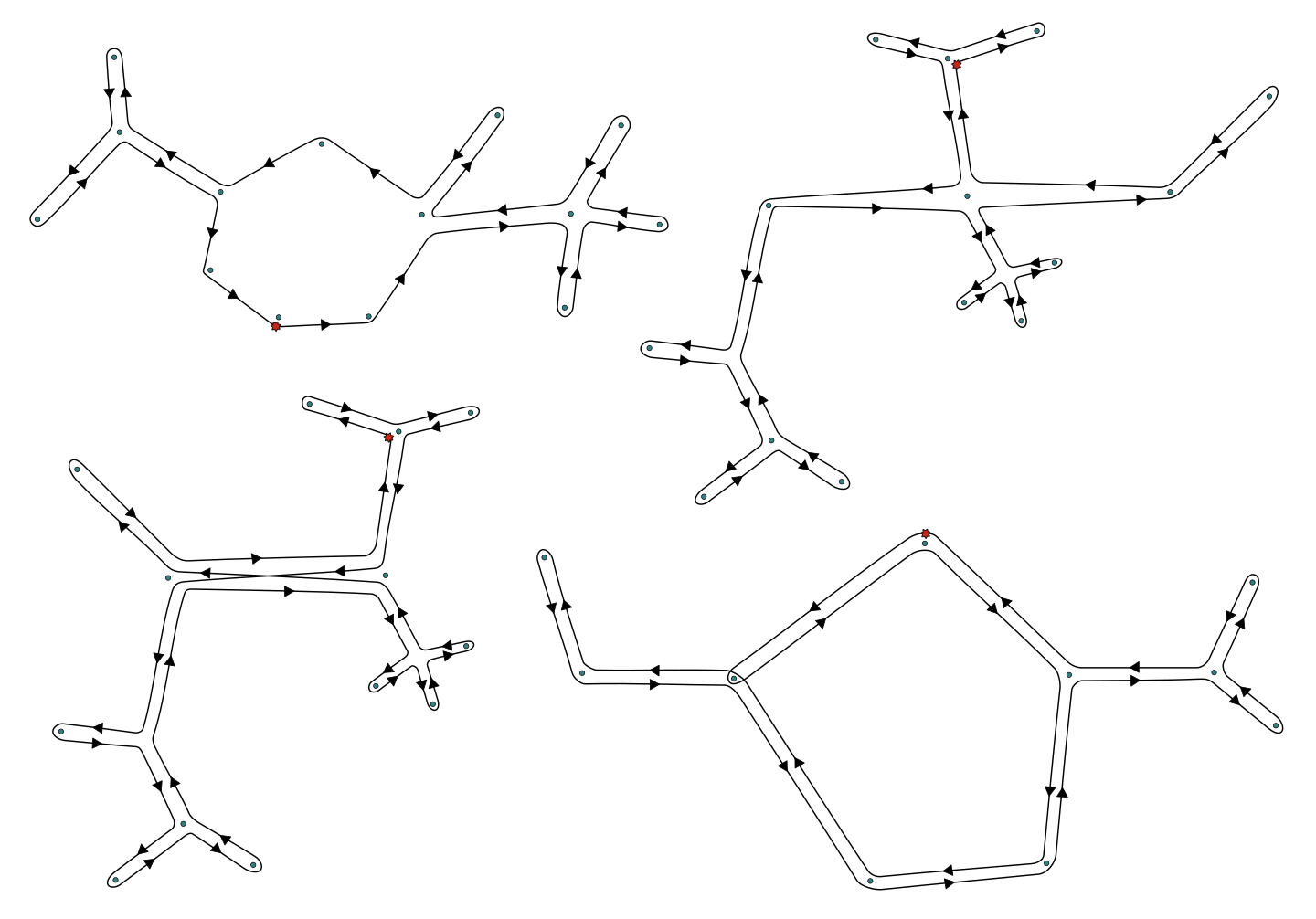} 
  \caption{Graphs for which $n^{-k/2}\mathcal A_\pi^{(n)}$ has a non--trivial limit. 
  Top left, top right, bottom left, bottom right correspond respectively to cases 1, 2, 3 and 4 above. Green dots denote the vertices and red stars denote initial points.} 
\end{figure}

\noindent
Before turning to the convergence of the expectation and 
the fluctuations of traces, we state a general lemma which will be used 
frequently. 

\begin{lemma}[Even multiplicities in trees]
	\label{lemma:even_multi_trees}
	Let $\psi: [k] \rightarrow [n]$ be such that $\overline{G}_\psi$ 
	is a tree. Then, each edge of $\overline{G}_\psi$ has an 
	even multiplicity in $G_\psi$ with an equal number of oriented edges 
	in each direction. 
\end{lemma}

\begin{proof}
	Let $e = (i, i+1) \in E_\psi$ be an edge. 
	Denote $u = \psi(i)$ 
	and $v = \psi(i+1)$ its endpoints. 
	Since $\psi$ traces a cycle, there exists a sequence of edges 
	$(i+1, i+2), \dots, (i+r, i+r+1)$ such that $s((i+1, i+2)) = v$ and 
	$t((i+r, i+r+1)) = u$ for some $1 \leq r \leq k-1$. Take such $r$ minimal. 
	Since $\overline{G}_\psi$ is a tree, 
	one must have $s((i+r, i+r+1)) = v$ otherwise the graph $\overline{G}_\psi$ 
	would have a cycle. 
	Removing the edges $(i, i+1), \dots, (i+r, i+r+1)$ from $G_\psi$ 
	shows the result by induction on 
	$\left| \{ e \in E_\psi : \{ s(e), t(e) \} = \{ u, v \}  \} \right|$.
\end{proof}

\subsubsection{Convergence of expected value}
\label{subsection:proof_lemma_expectation_convergence}

Recall that
\[
	U_{k}^{(n)} = \mathrm{Tr} \left[ P_k \left(\frac{A_{n}}{\sqrt{n}} \right) \right] + nt \delta_{k=2}.
\]
as defined previously in \eqref{eq:Uk}. 
The main result of this section is Proposition \ref{prop:expectation_convergence} 
which shows the convergence of $E\big[ U_k^{(n)} \big]$ for each $k \geq 1$. 
\begin{proposition}[Vanishing Chebyshev expectation] For every $k \geq 1$,
	\label{prop:expectation_convergence}
		\[ \lim_{n \to \infty} \mathbb E\big[ U_k^{(n)} \big] = 0 .\]
\end{proposition}

\noindent
The proof is based on the convergence of traces of monomials which is Lemma 
\ref{lemma:expectation_monomial}.
As for Wigner matrices, Catalan numbers
$C_p = \frac{1}{p+1} \binom{2p}{p}$ are involved.

\begin{lemma}[Monomial expectation]
	\label{lemma:expectation_monomial}
	Let $p \geq 1$. Then, 
	\begin{itemize}
		\item[(i)] $\E \left[ \mathrm{Tr} \left[ \left(\frac{A_{n}}{\sqrt{n}} \right)^{2p} \right] \right] 
		- n C_{p} t^p  \xrightarrow[n \to \infty]{} 0$
		 and \\
		\item[(ii)] $\E \left[ \mathrm{Tr} \left[ \left(\frac{A_{n}}{\sqrt{n}} \right)^{2p+1} \right] \right] \xrightarrow[n \to \infty]{} 0$.
	\end{itemize}
\end{lemma}

\noindent
We now prove Proposition \ref{prop:expectation_convergence}, using  
the asymptotics of Lemma \ref{lemma:expectation_monomial}.

\begin{proof}[Proof of Proposition \ref{prop:expectation_convergence}]
	
	We begin by noticing that, if $\sigma_t$ denotes
	the uniform probability measure on the ellipse $\mathcal E_t$
	defined in \eqref{eq:ellipse},
	\[\int_{\mathcal E_t}
	w^{2p} \mathrm d \sigma_t(w)
	=
	\frac{1}{p+1}
	{2p \choose p} t^p = C_p t^p
	\]
	which can be obtained by using elliptic coordinates. Then, from 
	Lemma \ref{lemma:expectation_monomial}, $(i)$, the asymptotic
	\[ \mathbb E \left[ \mathrm{Tr} \left[ \left(\frac{A_{n}}{\sqrt{n}} \right)^{2p} \right] \right] 
	= n C_{p} t^p  + o(1)
	= n\int_{\mathcal E_t}
	w^{2p} \mathrm d \sigma_t(w) + o(1) \]
	implies that
	\[ \mathbb E \left[ \mathrm{Tr} \left[ P_k^{(t)}\bigg(\frac{A_{n}}{\sqrt{n}} 
	\bigg) \right] \right] 
	= n\int_{\mathcal E_t}
	P_k^{(t)}(w) \mathrm d \sigma_t(w) + o(1). \]
	It remains to compute 
	$\int_{\mathcal E_t}
	P_k^{(t)}(w) \mathrm d \sigma_t(w)$. To achieve this,
	we define
	\[
	\mathcal M(z) = \int_{\mathcal E_\rho}
	\log \left(1+t z^2 - z w \right)
	\mathrm d \sigma_t(w)  \]
	which is holomorphic for
	$z \in \mathbb D$ with
	\[
	\frac{d^k}{dz^k} \mathcal M(z) = \int_{\mathcal E_t}
	\frac{d^k}{dz^k}  \log \left(1+t z^2 - z w \right)
	\mathrm d \sigma_t(w) . \]
	In particular, the  coefficient $z^k$ of $\mathcal M(z)$
	is the integral of the coefficient
	$z^k$ of $\log \left(1+t z^2 - z w \right)$, 
	\[
	[z^k] \mathcal M(z) = -\int_{\mathcal E_t}
	\frac{P_k^{(t)}(w)}{k}\mathrm d \sigma_t(w).
	\]
	To make a connection with
	classical Hermitian random matrix theory we can compute
	\[S_t(z)=\int_{\mathcal E_t} \frac{\mathrm d \sigma_t(w)}{z-w}
	=
	\frac{1}{z}
	\sum_{p=0}^\infty 
	\left(\frac{\sqrt t}{z}\right)^{2p}
	C_p = \frac{1}{\sqrt t} h \left(\frac{z}{\sqrt t}
	\right),\]
	where $h (z)=
	\frac{z - \sqrt {z^2 - 4}}{2}$ is the holomorphic solution to
	$z=h+\frac{1}{h}$
	that
	goes to zero at infinity,
	which we may recognize as the
	Cauchy-Stieltjes transform
	of the semi-circle distribution (case $t=1$). 
	In particular, 
	$S_t(z^{-1} + tz ) = z$ and
	we may connect the function $S_t$ with $\mathcal M$  by taking its derivative
	\begin{align*}\frac{\mathrm d}{\mathrm d z}
		\mathcal M(z)
		&=\frac{1}{z}
		+\left(-\frac{1}{z^2} + t\right)
		\int_{\mathcal E_t}
		\frac{1}{\frac{1}{z} + t z - w}
		\mathrm d \sigma_t (w)					
		=\frac{1}{z}
		+\left(-\frac{1}{z^2} + t\right)
		S_t\left(\frac{1}{z} + t z\right)		\\
		&= 
		\frac{1}{z}
		+\left(-\frac{1}{z^2} + t\right)
		z											
		=t z.
	\end{align*}
	So, since we also know that
	$\mathcal M(0) = 0$, we obtain
	$\mathcal M(z) = \frac{t z^2}{2}$, which implies that
	\[\int_{\mathcal E_t}
	P_k^{(t)}(w)\mathrm d \sigma_t(w) = 0
	\mbox{ for } k \neq 2 \quad \mbox{ while }
	\int_{\mathcal E_t}
	P_2^{(t)}(w)\mathrm d \sigma_t(w) = -t.\]

\end{proof}

\noindent
We now turn to the proof of Lemma \ref{lemma:expectation_monomial}.

\begin{proof}[Proof of Lemma \ref{lemma:expectation_monomial}]

		For $k \geq 1$, write
		\begin{equation*}
			\E
			\bigg[ \mathrm{Tr} \bigg[ \bigg(\frac{A_{n}}{\sqrt{n}} \bigg)^k \bigg] \bigg]=
			 n^{-k/2} \sum_{\pi \in \mathcal P(k)} 
			\mathbb E[\mathcal A_\pi^{(n)}]
			= n^{-k/2} \sum_{\pi \in \mathcal P(k)} 
			C_\pi^{(n)} \alpha_\pi,
		\end{equation*}
		with $\alpha_\pi$ being the common value of 
		$\mathbb E[a_\psi]$ for $\psi$ inducing the partition $\pi$
		and $C_\pi^{(n)}$ is the number of $\psi:[k] \to [n]$ that induce
		$\pi$.
		Since the choice of $\psi$ inducing $\pi$
		amounts to choosing the image of each block of $\pi$,
		we find that, for $n \geq k$,  $C_\pi^{(n)} = n^{\underline {|\pi|}}$, where 
		$|\pi|$ denotes the number of blocks of $\pi$
		and $n^{\underline k} = n(n-1)\dots(n-k+1)$ is the falling factorial.
		Let us investigate the cases where $\alpha_\pi \neq 0$.\\
		\\
		As entries are centered, for $\alpha_\pi$ not to vanish, every edge of
		$G_\pi$ has to be at least double. 
		Let us denote this set by
		$\mathcal D_k = \{\pi \in \mathcal P(k): G_\pi \mbox{ 
			has no simple edges}\}$.
		Since for $\pi \in \mathcal D_k$  we have
		$2|\overline E_\pi| \leq |E_\pi|$, we obtain
		$|V_\pi| \leq |\overline E_\pi| + 1 \leq \frac{k}{2}+1$ so that,
		for $n \geq k$,
		\begin{equation*}
			\E
			\bigg[ \mathrm{Tr} \bigg[ \bigg(\frac{A_{n}}{\sqrt{n}} \bigg)^k \bigg] \bigg]= n^{-k/2} \sum_{\ell =1 }^{\lfloor \frac{k}{2}+1 \rfloor}
			n^{\underline \ell}
			\sum_{\pi \in \mathcal D_k, |\pi| = \ell} 
			\alpha_\pi
		\end{equation*}
		For odd values of $k$, the only possible contribution to the limit comes
		from $\ell = (k+1)/2$.
		For this $\ell$, we
		would have $|V_\pi| =(k+1)/2$ and, since $|V_\pi| -1 \leq |\overline E_\pi| \leq k/2$, we must
		have $|\overline E_\pi| = (k-1)/2$ so that
		$\overline{G}_\pi$ is a tree. 
		The fact that $|\overline E_\pi| = (k-1)/2$ and that each edge is at least
		double in $E_\pi$ tells us that 
		there is exactly one edge that is triple in $E_\pi$.
		This would contradict the result of Lemma \ref{lemma:even_multi_trees}
		so that there
		is no $\pi \in \mathcal D_k$ satisfying
		$|\pi| = (k+1)/2$.
		Then, for odd $k$ the expected value goes to zero.\\
		\\
		For 
		even values $k=2p$,
		\begin{equation*}
			\E
			\bigg[ \mathrm{Tr} \bigg[ \bigg(\frac{A_{n}}{\sqrt{n}} \bigg)^{2p} \bigg] \bigg]=  
			n\hspace{-1mm}\sum_{\pi \in \mathcal D_{2p}, |\pi| = p+1} 
			\hspace{-3mm}\alpha_\pi 
			+\bigg[-\binom{p+1}{2}
			\sum_{\pi \in \mathcal D_{2p}, |\pi| = p+1} 
			\hspace{-3mm} \alpha_\pi 
			+\hspace{-1mm}	\sum_{\pi \in \mathcal D_{2p}, |\pi| = p} 
			\hspace{-3mm} \alpha_\pi \bigg]
			+ O\Big(\frac{1}{n}\Big).
		\end{equation*}
		For $\pi \in \mathcal D_{2p}$ satisfying $|\pi| = p+1$,
		the graph
		$\overline G_\pi$ has $p$ edges 
		and $p+1$ vertices so that it is a tree and $\pi$ draws a path on $\overline G_\pi$
		with exactly two edges passing through each edge in $\overline E_\pi$,
		necessarily once in each direction according to Lemma \ref{lemma:even_multi_trees}.
		Denote by $\mathcal T_p$ the set of these path patterns, also
		known as rooted plane trees.
		Then, $\alpha_\pi = t^p$ for every $\pi \in \mathcal T_p$ so that
		\[\sum_{\pi \in \mathcal D_{2p}, |\pi| = p+1} \alpha_\pi = |\mathcal T_p| t^p .\]
		For $\pi \in \mathcal D_{2p}$ satisfying $|\pi| = p$, we
		can either have $|\overline E_\pi| = p$ or $p-1$.
		\begin{itemize} 
			\item If $|\overline E_\pi| = p$, then
			$\overline G_{\pi}$ is a graph with $p$ edges and $p$ vertices
			so that it is a unicyclic graph.
			For $\alpha_\pi$ to be non-zero, 
			the path drawn by $\pi$ on $\overline G_\pi$
			has to have
			exactly two edges passing in opposite directions through each edge in $\overline E_\pi$ (its cycle can be a loop).
			The value of $\alpha_\pi$ in this case is 
			$\mathbb E[a_{12}a_{21}]^p = t^p$ in the
			case of no loops and
			$\mathbb E[a_{12}a_{21}]^{p-1}\mathbb E[a_{11}^2] = t^{p-1}t$
			in the case where there is a loop which turns out to be the same number.
			We denote by $\mathcal N_p$ the set of all these path patterns.
			\item If $|\overline E_\pi| = p-1$, 
			then
			$\overline G_{\pi}$ is a graph with $p-1$ edges and $p$ vertices
			so that it is a tree.
			Since a closed path drawn on a tree should pass
			an even number of times by each edge with 
			the same number of times in each direction 
			by Lemma \ref{lemma:even_multi_trees},
			every edge of $\overline G_{\pi}$
			should be double in $G_\pi$ except only for one edge of multiplicity four.
			The value of $\alpha_\pi$ in this case is
			$\mathbb E[(a_{12}a_{21})^2]\mathbb E[a_{12}a_{21}]^{p-2} = 
			(2t^2) t^{p-2} = 2t^p$ and
			we denote by $\mathcal N_p'$ the set of these path patterns.
		\end{itemize}
		
		\noindent
		We will now show that the $O(1)$ term in the asymptotic development 
		of traces vanishes by showing the combinatorial equality
		\begin{equation}
			\label{eq:merging_trees}
			\binom{p+1}{2} |\mathcal T_p| = |\mathcal N_p| + 2 |\mathcal N'_p|.
		\end{equation}
		The idea behind \eqref{eq:merging_trees} is that,
		given an element of $\mathcal T_p$, one can choose two vertices
		and identify them to obtain an element either of
		$\mathcal N_p$ if the vertices are not at
		distance two or of $\mathcal N_p'$ if
		the chosen vertices are at distance two.
		More precisely, consider
		\[\Phi:\bigsqcup_{\pi \in \mathcal T_p}
		\{\{a,b\} \subset \pi: a \neq b\} \to \mathcal N_p 
		\cup \mathcal N_p'\]
		defined by taking $\pi \in \mathcal T_p$ and
		a pair of blocs $a_1,a_2$ of $\pi = \{a_1,a_2,a_3,\dots,a_{p+1}\}$ to the partition
		$\{ a_1\cup a_2, a_3, \dots, a_{p+1}\}$.
		We may notice that an element of $\mathcal N_p$
		has a unique preimage by $\Phi$.
		Here is a possible way to see this.
		Let $\pi \in \mathcal N_p$ and let $\psi: [k] \rightarrow [n]$ 
		be a function which induces $\pi$. Consider the smallest index 
		$\ell \in [k]$ such that
		\begin{itemize}
			\item edges $(\ell, \ell+1)$ and $(\ell+r, \ell+r+1)$ for some $r \geqslant 1$ 
			are in opposite directions and belong to the cycle of $\overline{G}_\pi$, 
			and
			\item edges $(\ell+1, \ell+2), \dots, (\ell + r-1, \ell + r)$ 
			belong to a tree component in $\overline{G}_\pi$.
		\end{itemize}
		Then, as $\overline{G}_\pi$ 
		has a cycle, there exists $s \in [k]$ such that 
		$\psi(s) \neq \psi(\ell)$ and $\psi(s+1) = \psi(\ell+1)$. 
		In $\pi$, there is the block 
		\begin{equation*}
			\{ t_1, \dots, t_p, \ell+1, \ell+r, s+1, t_{p+1}, \dots, t_{q} \}
		\end{equation*}
		for some $t_1, \dots, t_q \in [k]$ and $q \geqslant 0$. 
		The preimage of $\pi$ is given by splitting the above block into two 
		distinct blocks: 
		\begin{equation*}
			\{ t_1, \dots, t_p, \ell+1, \ell+r \} \sqcup \{s+1, t_{p+1}, \dots, t_{q} \}.
		\end{equation*}

\begin{figure}[ht] 
	\centering 
	\includegraphics[scale=0.6]{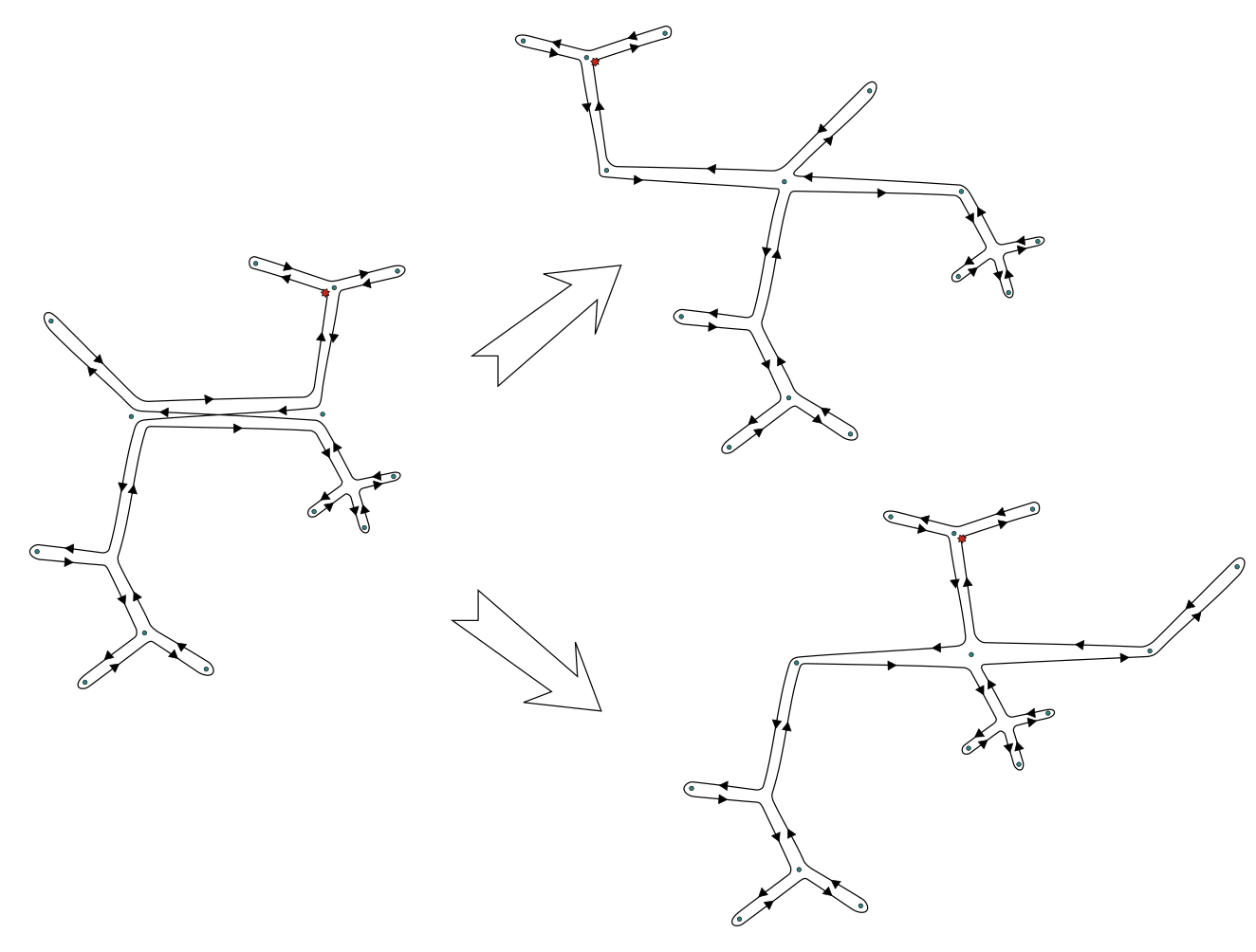} 
	\caption{A tree with a quadruple edge with the two possible ways
		of ungluing it to obtain a planar tree.} 
	\label{fig:ungluing}
\end{figure}

On the other hand, an element of $\mathcal N'_p$ has 
exactly two preimages by $\Phi$ each one being obtained
by splitting blocks 
one of
the two endpoints of the ``quadruple'' edge (see Figure \ref{fig:ungluing}).
Then,
\[\binom{p+1}{2} |\mathcal T_p| = |\mathcal N_p| + 2 |\mathcal N'_p|\]
so that 
\begin{equation*}
	\E
	\bigg[ \mathrm{Tr} \bigg[ \bigg(\frac{A_{n}}{\sqrt{n}} \bigg)^{2p} \bigg] \bigg]=  n|\mathcal T_p| t^p
	+\bigg[\binom{p+1}{2}|\mathcal T_p| t^p - |\mathcal N_p| t^p - 
	|\mathcal N_p'| 2 t^p\bigg]
	+ O\Big(\frac{1}{n}\Big)
	=nC_p t^p + O\Big(\frac{1}{n}\Big).
\end{equation*}

\end{proof}

\subsubsection{Convergence of fluctuations}
\label{sub:Fluctuations}

Recall the definition $\mathcal A_\pi^{(n)} = 
\hspace{-1mm}\sum \limits_{\substack{\psi:[k] \to [n] \vspace{0.5mm}  
\\ \pi_\psi = \pi}} a_\psi$ from \eqref{eq:Api}
and let us use the following notation. 

\begin{definition}[Rooted plane tree and 
	uniciclic graph] 
	\label{def:rooted_plane_tree_and_unicyclic}
	Let $k \geq 1$ and $\pi \in \mathcal P(k)$.
	\begin{itemize}
	 \item	We say that $\pi \in \mathcal P(k)$ is a \textit{rooted plane tree} if
	$\overline G_\pi$ is a tree and each edge in $G_\pi$ is double,
	once in each direction.
	\item We say that $\pi \in \mathcal P(k)$ is a \textit{rooted plane unicyclic graph} if
	$\overline G_\pi$ has a unique cycle, with each edge in 
	$G_\pi$ being simple if it belongs to the unique cycle and double, 
	once in each direction, if it does not belong to the unique cycle.
	\end{itemize}
	\end{definition}

\begin{proposition}[Gaussian limit process] 
\label{prop:GaussianLimit}
	The family $(n^{-k/2}\mathcal A_\pi^{(n)} - n^{-k/2} 
	\mathbb E[\mathcal A_\pi^{(n)}])_{\pi \in \mathcal P(k),k \geq 1}$
	converges to a Gaussian process as $n$ goes to infinity.
	Moreover, the only $\pi$ that give non-trivial limits
	are the rooted plane trees and the rooted plane unicyclic graphs, i.e.,
	if $\pi$ is neither of those, the
	sequence $n^{-k/2}\mathcal A_\pi^{(n)} - n^{-k/2}\mathbb E[\mathcal A_\pi^{(n)}]$
	converges to zero in law.
\end{proposition}

\begin{proof}
	Let $k_1,\dots,k_\ell \geq 1$
	and denote $k = k_1+\dots+k_\ell$. For each
	$i \in \{1,\dots,\ell\}$ choose
	$\pi_i \in \mathcal P(k_i)$ and an ``exponent'' $s_i \in \{\cdot, *\}$,
	and use the convention  $x^{(\cdot)} = x$ and $x^{(*)} = \overline x$.
	Our goal is to show the limit
	\[\lim_{n \to \infty} 
	n^{-k/2}\mathbb E\bigg[ \prod_{i=1}^\ell (\mathcal A_{\pi_i} - \E[\mathcal A_{\pi_i}])^{(s_i)}  \bigg] 
	= \mathbb E \bigg[\prod_{i=1}^\ell Y_{i}^{(s_i)} \bigg].\]
	where $(Y_i)_{1\leq i \leq \ell}$ is a Gaussian vector on
	$\mathbb C^\ell$. Let us write
	\[	n^{-k/2}\mathbb E\bigg[ \prod_{i=1}^\ell (\mathcal A_{\pi_i} - \E[\mathcal A_{\pi_i}])^{(s_i)}  \bigg] 
	=
	n^{-k/2}
	\sum_{\substack{\psi:[\mathbf k] \to [n] \vspace{0.5mm}  \\ \pi_{\psi_j} = \pi_j}}
	\mathbb E\bigg[ \prod_{j=1}^\ell (a_{\psi_j} - \mathbb E[a_{\psi_j}])^{(s_j)}  \bigg],
	\]
	where we recall that $[\mathbf k] = \sqcup_{j=1}^\ell [k_j]$
	and $\psi = \sqcup_{j=1}^\ell \psi_j$ with
	$\psi_j:[k_j] \to [n]$.
	Since the expected value
		$\mathbb E[ \prod_{j=1}^\ell (a_{\psi_j} - \mathbb E[a_{\psi_j}])^{(s_j)} ]$
		depends only on
		the partition induced by $\psi$, we can write
		\[ 
		\sum_{\substack{\psi:[\mathbf k] \to [n] \vspace{0.5mm}  \\ \pi_{\psi_j} = \pi_j}}
		\mathbb E\bigg[ \prod_{j=1}^\ell (a_{\psi_j} - \mathbb E[a_{\psi_j}])^{(s_j)}  \bigg]
		=
		\sum_{\substack{\tau \in \mathcal P(\mathbf k) \vspace{0.5mm}  \\ 
				\tau_j = \pi_j}}  C_\tau^{(n)} \beta_\tau
		\]
	with the following notation. The partition
	$\tau_j \in  \mathcal P(k_j)$ is
	the one induced by $\tau$ via the canonical inclusion
	$I_i:[k_j] \to [\mathbf k]$, i.e., the part of the path pattern
	traced by the $j$-th path, $\beta_\tau = 
	\E[ \prod_{j=1}^\ell (a_{\psi_j} - \mathbb E[a_{\psi_j}])^{(s_j)}  ]$
	for any $\psi$ inducing $\tau$ and
	$C_\tau^{(n)}$ is the cardinal of the set
	$\{\psi:[\mathbf k] \to [n]: \psi \mbox{ induces } \tau \}$
	which equals $n^{\underline{|\tau|}} = n(n-1)\dots(n-|\tau| + 1)$
	if $n \geq k$. \\
	\\
	The question now reduces to understand
	that $n^{-k/2}C_\tau^{(n)}$ has a limit if $\beta_\tau \neq 0$
	and to understand for which of those $\tau$ the limit is not zero. 
	This is the purpose of Lemmas \ref{lem:ShareEdge} and 
	\ref{lem:BoundTau} that we state and prove now. The end of the proof of 
	Proposition \ref{prop:GaussianLimit} will follow afterwards.

\begin{lemma} \label{lem:ShareEdge}
	If $\beta_\tau \neq 0$ then
	each edge 
	of $\overline G_\tau$ is at least double in $G_\tau$
	and for each $i \in \{1,\dots,\ell\}$ there is $j \in \{1,\dots,\ell\}$
	different from $i$ such that
	$\overline G_{\tau_i}$ and $\overline G_{\tau_j}$ share an edge.
\end{lemma}
\begin{proof}
	Let $\psi:[\mathbf k] \to [n]$ which induces $\tau$.
	If $G_\psi$ has a simple edge
	then there is a unique $i \in \{1,\dots,\ell\}$ 
	such that the edge $(r, r+1) \in E_i$ of $G_{\psi_i}$ 
	for some $r \in [k_i]$ is simple.
	In that case $\mathbb E[a_{\psi_i}] = 0$ since $G_{\psi_i}$
	has a simple edge and 
	since $a_{\psi_i(r), \psi_i(r+1)}$ is centered,
	\[\beta_\tau =
	\E\bigg[\prod_{j=1}^\ell (a_{\psi_j} - \mathbb E[a_{\psi_j}])^{(s_j)}\bigg]
	= \E \left[a_{\psi_i(r), \psi_i(r+1)}^{(s_i)} \right]
	\E\bigg[\prod_{j\neq i} (a_{\psi_j} - \mathbb E[a_{\psi_j}])^{(s_j)}\bigg]
	=0.\]
	Let $i \in \{1,\dots,\ell\}$.
	If there does not exist such index $j$, 
	the random variable $(a_{\psi_i} - \mathbb E[a_{\psi_i}])^{(s_i)}$
	would be independent of the product
	$\prod_{j\neq i} (a_{\psi_j} - \mathbb E[a_{\psi_j}])^{(s_j)}$ so that,
	since $(a_{\psi_i} - \mathbb E[a_{\psi_i}])^{(s_i)}$ is centered,
	$\beta_\tau$ would be zero.
\end{proof}

\vspace{1.5mm}

\noindent
	Recall that $k=k_1+\dots+k_\ell$.

\begin{lemma} 
	\label{lem:BoundTau}
	Suppose that $\beta_\tau \neq 0$.
	Then, $|\tau| \leq k/2$. Moreover,
	if some connected component of $\overline G_\tau$
	involves three or more $\overline G_{\tau_i}$, then the strict inequality
	$|\tau| < k/2$ holds.
\end{lemma}
	\begin{proof}
		
		We will actually show the this lemma under the conclusions of Lemma 
		\ref{lem:ShareEdge}.\\
		\\
		Recall that $|\tau|$ counts the number of vertices in $G_\tau$ so that we
		want to show that $|V_\tau| \leq k/2$.
		A connected component of 
		$\overline{G}_\tau$ is formed by some $G_{\tau_{i_1}},\dots,G_{\tau_{i_s}}$
		so that
		it is enough to prove this inequality for $\tau$ restricted to
		$[k_{i_1}]\sqcup \dots \sqcup [k_{i_s}]$.
		In other words, we may assume without loss of generality 
		that $\overline G_\tau$ is connected.\\
		\\
		Since each edge is at least double we have that
		$|\overline E_\tau| \leq k/2$.
		The inequality 
		$|V_\tau| \leq |\overline E_\tau| + 1$
		tells us that $|V_\tau| \leq k/2 + 1$
		and we need to understand why $|V_\tau|$ cannot be in $(k/2,k/2+1]$.
		\begin{itemize}
			\item[--] If $k$ is even and $|V_\tau| = k/2+1$ then 
		$|\overline E_\tau|=k/2$ so that $\overline G_\tau$ is a tree
		and each edge is double.
		Since
		$\tau_i$ traces a closed path in this tree, it must
		traverse each edge twice, once in each direction 
		by Lemma \ref{lemma:even_multi_trees}.
		This implies that $\overline G_{\tau_i}$ does not share edges
		with any other $\overline G_{\tau_j}$ because each edge of
		$G_\tau$ is double which contradicts
		$\beta_\tau \neq 0$ by Lemma \ref{lem:ShareEdge}. \vspace{1mm}
			\item[--] If $k$ is odd and $|V_\tau| = k/2 + 1/2$ then
			$|\overline E_\tau| = k/2-1/2$ so that again $\overline G_\tau$
			is a tree.
			But, using Lemma \ref{lemma:even_multi_trees}, 
			a closed path in a tree contains an even number of edges
			so that each $k_i$ has to be even which contradicts that
			$k = k_1+\dots + k_\ell$ is odd.
		\end{itemize}
		If $|\tau| = k/2$ then
		$|\overline E|$ can be either $k/2$ or $k/2-1$.
		The first case happens when $\overline G_\tau$ is unicyclic
		and the second case happens when it is a tree.
		We need to see why in this case $\ell$ must be $2$,
		i.e., why there can be only two $\overline G_{\tau_i}$ forming $\overline G_\tau$.
		
		\begin{itemize}
			\item[--] Suppose that $|\tau|=|\overline E| = k/2$
			so that $\overline G_\tau$ is unicyclic.
			Then, 
			each $\tau_i$ has to traverse each edge of the unique cycle at least once
			because, if not, 
			$\tau_i$ would draw a path on a tree ($\overline G_\tau$ with that edge removed) 
			so that each edge would be double
			and, since each edge of $G_\tau$ is precisely double, 
			$G_{\tau_i}$ would not share edges
			with the other $G_{\tau_j}$. 
			But there cannot be
			three $G_{\tau_i}$ passing through the unique cycle because
			each edge is double so that
			there has to be exactly two $G_{\tau_i}$.
			In particular, $G_{\tau_i}$ is also unicyclic and 
			$G_\tau$ is made of $G_{\tau_i}$ and $G_{\tau_j}$ by
			gluing them along the unique cycle because, if an
			edge is traversed at least once by $\tau_i$ then it is traversed
			exactly twice, once in each direction
			(if not, we would be able to form a cycle passing through
			this edge).\vspace{1mm}
			\item[--] Suppose that $|\tau|=|\overline E| +1= k/2$
			so that $\overline G_\tau$ is a tree.
			In this case, this implies in particular, that $\overline G_{\tau_i}$ is a tree
			with each of its edges traversed twice by $\tau_i$, once in 
			each direction. Take $\tau_i$ and $\tau_j$ for $i \neq j$
			such that $G_{\tau_i}$ and $G_{\tau_j}$ share an edge.
			This implies that
			this shared edge is at least ``quadruple''. Since
			$|\overline E| = k/2 - 1$, $|E| = k/2$ and each edge is at least double
			we must have that every edge of $G_{\tau}$ is
			precisely double except for the
			``quadruple edge. This implies that any other 
			$\overline G_{\tau_p}$
			cannot share an edge with any other $G_{\tau_q}$
			which cannot happen by Lemma \ref{lem:ShareEdge}.
			So, $G_\tau$ is made of two trees $G_{\tau_i}$
			and $G_{\tau_j}$ by gluing them along an edge
			(which may be thought of as a degenerate cycle).
			
		\end{itemize}
		
	\end{proof}
	
	\noindent
	We turn back to the proof of Proposition \ref{prop:GaussianLimit}. 
	Lemmas \ref{lem:ShareEdge} and \ref{lem:BoundTau}
	already show the Gaussian behavior.
	Indeed, let us denote by $\mathcal P_2(\pi_1,\dots,\pi_\ell)$
	the set of partitions $\tau \in \mathcal P(\mathbf k)$
	such that $\tau_i = \pi_i$ for every $i \in [\ell]$ and 
	such that $\beta_\tau \neq 0$, 
	i.e., that satisfy the conditions of Lemmas \ref{lem:ShareEdge} and \ref{lem:BoundTau}.
	The proof of Lemma \ref{lem:BoundTau} showed that connected components of $G_\tau$ 
	come from two $\tau_i$'s which are paired according to a common cycle or 
	via an edge which will be quadruple if both are rooted plane trees.
	Then, by Lemma \ref{lem:BoundTau}, there exists a pair partition 
	that we denote $\Pi_\tau$ of $[\ell]$ 
	where a block is $\{i,j\}$ if some connected
	component of $G_\tau$ is formed by
	$G_{\tau_i}$ and $G_{\tau_j}$.
	In particular, $\ell$ has to be even for 
	 $\mathcal P_2(\pi_1,\dots,\pi_\ell)$ not to be empty. In all cases,
	
	\[ \lim_{n \to \infty}	
	n^{-k/2}\mathbb E\bigg[ \prod_{i=1}^\ell (\mathcal A_{\pi_i} - \E[\mathcal A_{\pi_i}])^{(s_i)}  \bigg] 
	= \hspace{-1mm} \sum_{\tau \in \mathcal P_2(\pi_1,\dots,\pi_\ell)} 
	\hspace{-1mm} \beta_\tau \]
	which would be zero if $\ell$ is odd.
	Denote by $\mathcal P_2(\ell)$ the set
	of pair partitions of $[\ell]$. For each partition $\tau  \in
	\mathcal P_2(\pi_1,\dots,\pi_\ell)$ consider the family
	of partitions $(\tau_a)_{a \in \Pi_\tau} $, where if 
	$a= \{i,j\}$ is a block of $\Pi_\tau$, the partition
	$\tau_a \in \mathcal P_2(\pi_i,\pi_j)$ is the one induced on 
	$[k_i] \sqcup [k_j]$ by $\tau$. We identify 
	$\mathcal P_2(\pi_i,\pi_j) \simeq \mathcal P_2(\pi_j,\pi_i)$
	so that the order is not important. The
	assignment $\tau \mapsto (\tau_a)_{a \in \Pi_\tau} $ defines a bijection
	\[\mathcal P_2(\pi_1,\dots,\pi_\ell) \to \bigsqcup_{\Pi  \in \mathcal P_2(\ell)} 
	\prod_{\{i,j\} \in \Pi}\mathcal P_2(\pi_i,\pi_j).\]
	Notice that, since the components of $G_\tau$ are independent,
	$\beta_\tau = \prod_{a \in \Pi_\tau} \beta_{\tau_a}$ where we defined
	$\beta_{\tau_a} = 
	\mathbb E[(\mathcal A_{\tau_i} - \mathbb E[\mathcal A_{\tau_i}])^{(s_i)}
	(\mathcal A_{\tau_j} - \mathbb E[\mathcal A_{\tau_j}])^{(s_j)}]$
	whenever $a = \{i,j\}$.
	Then, we may write
	\begin{align*}
		\sum_{\tau \in \mathcal P_2(\pi_1,\dots,\pi_\ell)} 
	\hspace{-1mm} \beta_\tau &= \sum_{\Pi \in \mathcal P_2(\ell)}
	\sum_{\tau: \Pi_\tau = \Pi} \beta_\tau \\
	&= \sum_{\Pi \in \mathcal P_2(\ell)}
	\sum_{\tau: \Pi_\tau = \Pi} \, \prod_{a \in \Pi} \beta_{\tau_a} \\
	&= \sum_{\Pi \in \mathcal P_2(\ell)} 
	\prod_{\{i, j\} \in \Pi} \left( \sum_{\tau \in \mathcal P_2(\pi_i,\pi_j)} \beta_{\tau} \right).
	\end{align*}
	where we have used the bijection described above together
	with the distributive property.
	By Isserlis--Wick's theorem,
	we know that
	\begin{equation}
		\label{eq:limit_exp_product}
		\E \bigg[ \prod_{i=1}^\ell Y_{i}^{(s_i)} \bigg] =
		\sum_{\Pi \in \mathcal P_2(\ell)}
	   \prod_{\{i,j\} \in \Pi} \sum_{\tau \in \mathcal P_2(\pi_i,\pi_j)} \beta_{\tau}
	\end{equation}
	if $(Y_1,\dots,Y_\ell)$ is a centered Gaussian vector with covariances
	\[\mathbb E[Y_i^{(s_i)} Y_j^{(s_j)}]
	= \sum_{\tau \in \mathcal P(\pi_i,\pi_j)} \beta_{\tau}.\]
	Since the right-hand side of \eqref{eq:limit_exp_product}
	is the limit of covariances,
	we have shown that the family $(n^{-k/2}\mathcal A_\pi^{(n)} - n^{-k/2} 
	\mathbb E[\mathcal A_\pi^{(n)}])_{\pi \in \mathcal P(k),k \geq 1}$
	converges to a Gaussian family as $n \to \infty$.
	In fact, in the proof of Lemma 
	\ref{lem:BoundTau} we already found what graphs will contribute and how.

	\begin{lemma}[Contributing graphs]
		\label{lem:Contributing_graphs}
		 If $\mathcal P_2(\pi_i,\pi_j)$ is non-empty then
		both $\pi_i$ and $\pi_j$ are either
		plane rooted trees or they are plane rooted unicyclic graphs
		whose cycles have the same length.
	\end{lemma}
	\begin{proof}
		The proof
		is contained in the proof of Lemma \ref{lem:BoundTau}.
	\end{proof}
	
	\noindent
	This concludes the proof of Proposition
	\ref{prop:GaussianLimit}. Figure 
	\ref{fig:paired_components} 
	gives an example of contributions described by 
	Lemma \ref{lem:Contributing_graphs}. 

	\begin{figure}
		\centering
		\includegraphics[scale=0.4]{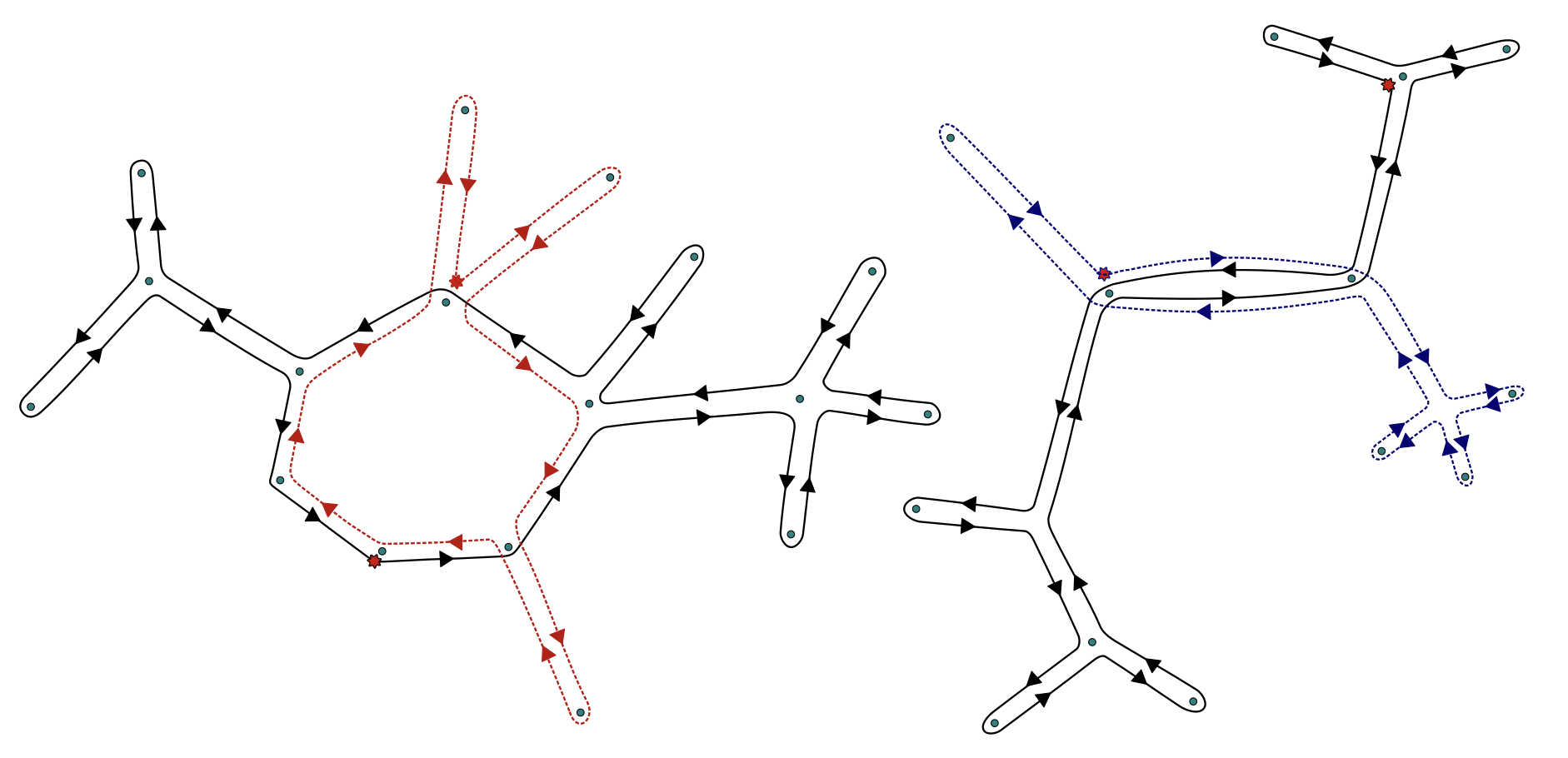}
		\caption{Two elements of $\mathcal{P}_2(\pi_1, \pi_2)$ where 
		$\pi_1, \pi_2$ are both rooted unicyclic graphs (left) or 
		both plane rooted trees (right).}
		\label{fig:paired_components}
	\end{figure}
	
\end{proof}

\begin{remark}
	A possible way to prove Proposition 
	\ref{proposition:convergence_to_gaussian} would be to compute explicitly 
	the limiting Gaussian family corresponding to traces of powers. These 
	Gaussians variables would not be independent. One would have to show that their  
	covariance is diagonalised by the Chebyshev polynomials thereby implying that the 
	variables $\left( U_k^{(n)} \right)_{k \geq 1}$ are independent.
	However, in Section \ref{subsec:variance_identification}, 
	we chose a different approach which provides 
	more intuition as to why the Chebyshev polynomials are the right 
	family for diagonalizing the covariance. 
\end{remark}

\subsubsection{Variance identification}
\label{subsec:variance_identification}

\begin{proof}[Proof of Proposition	\ref{proposition:convergence_to_gaussian}]

Let us denote by $\mathcal C_k^{(n)}$ the set
of $\psi: [k] \to [n]$ such that the associated graph $\pi_\psi$ 
is a rooted plane unicyclic graph
or a rooted plane tree. For such $\psi$, let 
denote $c(\psi)$ the 
set of edges in $E_\psi$ that are single, i.e, which form the cycle of 
$\overline{G}_\psi$. 
Note that this set is empty if $\pi_\psi$ is a rooted plane tree. 
Likewise, denote by $\ell(\psi)$ the set of edges in $E_\psi$ 
incident to a vertex of degree one in $\overline{G}_\psi$, 
that is, edges which are leaves in trees anchored on the cycle of $\overline{G}_\psi$.
Recall that

\[\Big( n^{-k/2}\mathrm{Tr}(A_n^k) 
- n^{-k/2}\mathbb E[\mathrm{Tr}(A_n^k)] \Big)
-n^{-k/2} \sum_{\psi \in \mathcal C_k^{(n)}}
\big( a_\psi - \mathbb E[a_\psi] \big) \xrightarrow[n \to \infty]{\mathrm{law}}0.	\]

\noindent
The idea is to ``regularize'' these terms and to consider instead, for any $\psi \in \mathcal C_k^{(n)}$,

\[\hat a_\psi := \prod_{e \in c(\psi)} a_e \prod_{e \in \ell(\psi)} (a_e a_{e^*} - t)\]

\noindent
where for an edge $e = (i, i+1)$, $a_e = a_{\psi(i), \psi(i+1)}$ and 
$e^* = (i+1, i)$.
A nice property of this term is that it has zero expected value
even if $\pi_\psi$ is a rooted plane tree.
Consider the sum 
\[:\hspace{-1mm} \mathrm{Tr}(A_n^k)\hspace{-1mm}: \hspace{2mm} 
\coloneqq \sum_{\psi \in \mathcal C_k^{(n)}} \hat a_\psi.\]

\noindent
By the same arguments as in Section \ref{sub:Fluctuations},
we know that
$(:\hspace{-1mm} \mathrm{Tr}(A_n^k)\hspace{-1mm}:)_{k \geq 1}$
converges to a Gaussian process as $n \to \infty$.
For a rooted plane unicyclic graph or a rooted plane tree $\pi$, we define

\[\hat{\mathcal A}_\pi^{(n)} \coloneqq \sum_{\substack{\psi:[k] \to [n] \vspace{0.5mm}  \\ \pi_\psi = \pi}} \hat a_\psi.\]
In the same way, $(\hat{\mathcal A}_\pi^{(n)})_\pi$
converges to a Gaussian process indexed by 
rooted plane unicyclic graphs and rooted plane trees.
Let us check that for any 
pair of rooted plane unicyclic graphs or trees $\pi_1$ and $\pi_2$,
\[\lim_{n \to \infty}
\mathbb E\big[\hat{\mathcal A}_{\pi_1}^{(n)} \hat{\mathcal A}_{\pi_2}^{(n)}\big]
= 0 \mbox{ and }
\lim_{n \to \infty}
\mathbb E\big[\hat{\mathcal A}_{\pi_1}^{(n)} \overline{\hat{\mathcal A}_{\pi_2}^{(n)} }\big]
= 0 \]
whenever one of them is not a cycle, 
that is, such that either $c(\psi_i) \neq E_{\psi_1}$ or $c(\psi_2) \neq E_{\psi_2}$ 
for any $\psi_1, \psi_2$ inducing $\pi_1$ and $\pi_2$ respectively. 
Since the leaves are centered in $\hat{\mathcal A}_{\pi_1}^{(n)}$ 
and $\hat{\mathcal A}_{\pi_2}^{(n)}$, 
the only possible non-vanishing contributions to the variance 
arise when $\pi_1$ and 
$\pi_2$ that either both have no leaves or share the same leaves. 
The case where both have no leaves corresponds 
exactly to the desired condition mentioned above.
Recall that the connected component must consist of double edges 
so that $\overline{G}_\pi$ has either a single cycle, 
or a quadruple edge in which case $\overline{G}_\pi$ is a tree. If $\pi_1$ and $\pi_2$ 
have a common leaf, this would result in a quadruple edge in $G_\pi$. 
Therefore, both $G_{\pi_1}$ and $G_{\pi_2}$ are plane rooted trees 
and the condition of having the same leaves implies that both 
$\overline{G}_{\pi_1}$ and $\overline{G}_{\pi_2}$ consist of only one edge. 
The latter case corresponds to two cycles of length two.
Therefore, the variances that will not have a zero limit
are those between two cycles of the same length so that, if $c_k$
denotes the partition $\{\{1\},\dots,\{k\}\}$, we have in law,

\[\lim_{n \to \infty} \Big(n^{-k/2}:\hspace{-1mm} 
\mathrm{Tr}(A_n^k)\hspace{-1mm}: \hspace{1mm}- \hspace{1mm} n^{-k/2}\hat{\mathcal A}_{c_k}\Big)=0.\]
If $Y_k$ denotes the limit in law of 
$n^{-k/2}\hat{\mathcal A}_{c_k}$, we would have, for $k \geq 3$,
\[\mathbb E[(Y_k)^2]= k \mathbb E[a_{12} a_{21}]^k  = k t^k  \quad \mbox { and } \quad 
\mathbb E[|Y_k|^2]=k \mathbb E[|a_{12}|^2]^k  = k.\]
For $k=2$, we have
$\mathbb E[(Y_2)^2] = 2 \mathbb E[(a_{12}a_{21}-t)^2] =2t^2$
and
$\mathbb E[(Y_2)^2] = 2 \mathbb E[|a_{12}a_{21}-t|^2] = 2$
and, for $k=1$,
$\mathbb E[(Y_1)^2] = \mathbb E[a_{11}^2] = t^2$ and $\mathbb E[|Y_1|^2]= 
\mathbb E[|a_{11}|^2] = 1$.
\\
\\
It now remains to express the normalised traces 
$:\hspace{-1mm} \mathrm{Tr}(A_n^k)\hspace{-1mm}: $
in terms of the traces of Chebyshev polynomials, that is, 
in terms of the variables $U_k^{(n)}$.
For simplicity, we forget the initial instant
of $\psi$, i.e., 
we consider the equivalence relation generated by
$\psi \sim \psi \circ \sigma$ with
$\sigma:[k] \to [k]$ given by $\sigma(i) = i+1$
and denote by $\mathcal D_k^{(n)}$ the 
quotient of $\mathcal C_k^{(n)}$ by this equivalence relation.
Notice that $a_\eta$ and $\hat a_\eta$ are well-defined for
$\eta \in \mathcal D_k^{(n)}$ and that, since each equivalence class
contains $k$ elements,
\[ :\hspace{-1mm} \mathrm{Tr}(A_n^k)\hspace{-1mm}: \ = 
\sum_{\psi \in \mathcal C_k^{(n)}} \hat a_\psi
=k \sum_{\eta \in \mathcal D_k^{(n)}} \hat a_\eta.\]
We may write
the sum of $\hat a_\eta$ over $\eta \in \mathcal D_k^{(n)}$
as a weighted sum of $a_\theta - \mathbb E[a_\theta]$ 
with $\theta \in \mathcal D_{k-2j}^{(n)} $
for $j \geq 0$ since each time we erase a leave we are erasing
two edges. 
Fix $j \geq 0$ and let us find the weight of
$\theta \in \mathcal D_{k-2j}^{(n)}$.
To construct an element from $\mathcal D_k^{(n)}$ by adding leaves
to $\theta$,
we would have to add $j$ of those leaves.
We should choose $j$ instants (with possible repetitions) among the $k-2j$ instants in the path
$\theta$ to introduce these leaves. This would give us $\binom{k-2j+j-1}{j}
= \binom{k-j-1}{j}$ choices.
Then, we would have
$(n-(k-2j))^{\underline j} \sim n^j $ 
choices for the vertices associated to the leaves. This means that the weight
of $a_\theta - \mathbb E[a_\theta]$ is
\[\binom{k-j-1}{j}(n-(k-2j))^{\underline j}(-t)^j.\]
Note that the case where 
$k-2j = 0$ is not relevant
because our variables $\hat a_\psi$ are centered. We can write
\begin{align*} 
	n^{-k/2}&\sum_{\psi \in \mathcal C_{k}^{(n)}} \hat a_\psi
=
n^{-k/2}k\sum_{0\leq j \leq k/2}
\binom{k-j-1}{j}(-t)^j (n-(k-2j))^{\underline j}
\sum_{\theta \in \mathcal D_{k-2j}^{(n)}}  (a_\theta - \mathbb E[a_\theta])
								\\
&=
\sum_{0\leq j \leq k/2}
\frac{k}{k-2j}\binom{k-j-1}{j}(-t)^j \frac{(n-(k-2j))^{\underline j}}{n^j}
\bigg( n^{-(k-2j)/2}\sum_{\psi \in \mathcal C_{k-2j}^{(n)}}  (a_\psi 
-\mathbb E[a_\psi]) 	\bigg).
\end{align*}

\noindent
Notice then that whenever $k-2j \neq 0$,
\[\frac{k}{k-2j}\binom{k-j-1}{j} (-t)^j
=\frac{k}{k-j}\binom{k-j}{j} (-t)^j= 
\alpha_{k-2j}^{(k)}\] 
where $\left( \alpha_{k-2j}^{(k)} \right)_j$ 
are the coefficients of $P_k$ given in \eqref{eq:expression_coefs_alpha}.
Since $\frac{(n-(k-2j))^{\underline j}}{n^j}$ converges to $1$, the limit
of
$n^{-k/2}\sum_{\psi \in \mathcal C_{k}^{(n)}} \hat a_\psi$
coincides with the limit of 
\[\mathrm{Tr}\Big( P_k \Big(\frac{A_n}{\sqrt n} \Big) \Big)
- \E\Big[\mathrm{Tr}\Big( P_k \Big(\frac{A_n}{\sqrt n} \Big) \Big)\Big] = 
U_k^{(n)} - \E[U_k^{(n)} ]. \]
More precisely, their difference goes to zero in law. 
Since 
$\E[U_k^{(n)}]$
goes to zero by Proposition
\ref{prop:expectation_convergence}, we obtain that
\[\lim_{n \to \infty} (U_k^{(n)} - n^{-k/2}\hat{\mathcal A}_{c_k})_{k \geq 1}
= 0
\]
which completes the proof of Proposition
	\ref{proposition:convergence_to_gaussian}.

\end{proof}

\section{Proof of Theorem \ref{th:cv_mean_poly_univ}}
\label{sec:proof_mean_carac_poly}

For $I \subset [n]$, we denote by $A_I$ the submatrix of $A_{n, t}$ 
obtained by taking rows and columns with index in $I$. 
Since one can write
\begin{equation}
	\label{eq:secular_coefs_dev}
	\det \left((1+tz^2) - z\frac{A_{n,t}}{\sqrt{n}} \right) 
	= \sum_{k=0}^{n} (1+tz^2)^{n-k} (-z/ \sqrt{n})^k S_k^{(n)}
\end{equation}
where $S_k^{(n)} = \sum_{I \subset [n]: |I|=k} \det(A_I)$ is coefficient of $w^k$ 
in the polynomial $\det(1+w A)$, 
the expectation $E[f_n(z)]$ depends only on $t$. 
It can be written in terms
of Hermite polynomials, see Proposition \ref{prop:mean_poly} below. 
We will not need this since
we can study its asymptotic behavior by considering
$A_{n,t}$ sampled from the Elliptic Ginibre ensemble of parameter $t$. \\
\\
Recall that if $ (X_n)_{n \geq 1}$ is a sequence of 
uniformly integrable random variables which converge in law to $X$, then 
$ \lim_{n \rightarrow \infty} \E[X_n] = \E [X]$. Since the second order moment 
of $\{ f_{n} \}_{n \geq 1}$ is uniformly bounded on compact subsets by Proposition 
\ref{proposition:uniform_control}, the family is uniformly integrable and therefore
 for $z \in \D$,
\begin{equation} \label{eq:ExpPolConvergence}
	\lim_{n \rightarrow \infty} \E[f_{n}(z)] = \E[e^{-F_t(z)}]
\end{equation}
with
\begin{equation*}
	F_t(z) = \sum_{k \geq 1} X_{k} \frac{z^k}{\sqrt{k}}
\end{equation*}
for a family $( X_{k} )_{k \geq 1}$ of independent Gaussian random variables 
satisfying $\E[X_{k}] = 0$, $\E[X_{k}^2] = t^k$ and $\E[|X_{k}|^2] = 1$.
Since $|\E[f_{n}(z)]| \leq \sqrt{\E[|f_{n}(z)|^2]} $
and since $\E[|f_{n}(z)|^2]$ is uniformly bounded on compact sets
by Proposition 
\ref{proposition:uniform_control},
$\mathbb E[f_{n}(z)]$ is a precompact sequence of holomorphic functions
by Montel's theorem. This implies that the convergence
in \eqref{eq:ExpPolConvergence} is uniform on compact sets.
It is enough to calculate
 \[\mathbb E[e^{-F_t(z)}]
 = \mathbb E\big[e^{-  \sum_{k \geq 1} X_{k} \frac{z^k}{\sqrt{k}}}\big]
 =
e^{  \frac{1}{2}\sum_{k \geq 1} \mathbb E[X_{k}^2] \frac{z^{2k}}{k}}
=
e^{  \frac{1}{2}\sum_{k \geq 1} t^k \frac{z^{2k}}{k}}
=\sqrt{1-t z^2}
 \]
which completes the proof of Theorem \ref{th:cv_mean_poly_univ}. \\

\noindent
For the sake of completeness, we provide an explicit expression of the 
expectation $\E[f_n(z)]$ in Proposition \ref{prop:mean_poly}.

\begin{proposition}[Mean characteristic polynomial]
	\label{prop:mean_poly}
	Let $A_{n,t}$ be a random matrix as in Theorem \ref{th:cv_mean_poly_univ}.
	Then, for every $z \in \D$, 
	\begin{equation}
        \label{eq:mean_hermite_expression}
        \E [f_{n,t}(z)] =  \ed^{-ntz^2/2} \left( \sqrt{\frac{t}{n}} z \right)^{n} 
         H_n \left( \sqrt{\frac{n}{t}} \left( \frac{1}{z} + tz \right) \right).
    \end{equation}
\end{proposition}

\begin{remark}[Universality of the expectation]
	\label{rem:UnivExp}
	Notice that the expectations involved in $\E[f_n(z)]$ 
	only depend on $\E[a_{1,2} a_{2,1}] = t$. 
	Therefore, $\E [f_{n,t}(z)]$ can be obtained by 
	considering EGE matrices for the same $t$ 
	and is related to the kernel
	when we see the eigenvalues as 
	a determinantal process.
	Namely, if $(Z_1,\dots,Z_n) \sim 
	\frac{1}{\mathcal Z}\prod_{i<j}^n|z_j - z_i|^2
	\mathrm d \mu^{\otimes_n}(z_1,\dots,z_n)$, then
	\[\mathbb E\big[\prod_{i=1}^n (z-Z_i) \big] = p_n(z),\]
	where $p_n$ is the monic orthogonal polynomial
	of degree $n$ with respect to $\mu$.
	\end{remark}

\noindent
For convenience of the reader, we prefer to give a more
direct proof of
Proposition \ref{prop:mean_poly}. Recall that 
$S_k^{(n)} = \sum_{I \subset [n]: |I|=k} \det(A_I)$ are the coefficients 
in \eqref{eq:secular_coefs_dev}.
	
	\begin{lemma}[Mean coefficient]
	\label{lem:expectation_secular}
		For $1 \leq k \leq n$,
		\begin{equation}
			\E[S_k^{(n)} ] = \begin{cases}
			&0 \text{ if } k \in 2\N +1 \\
			& \binom{n}{k} (k-1)!! (-t)^{k/2}  \text{ if } k \in 2 \N.
		\end{cases}
	\end{equation}
	where $(l)!! = l \cdot (l-2) \dots 3 \cdot 1$ for $l \in 2\N +1$.
	\end{lemma}
	
	\begin{proof}
		Let $I \subset [n]$ such that $|I| = k$. Then, 
		\begin{align*}
			\E [\det(A_I)] &= \sum_{\sigma} (-1)^{\sigma} \E \prod_{i \in I} a_{i, \sigma(i)}.
		\end{align*}
		where the sum is over permutations of $I$. The expectation is non-zero if and only if 
		$\sigma$ is a product of transpositions since each term is centered and $(a_{i,j}, a_{j,i})$ 
		is independent of the family $ \{ (a_{l,k}, a_{k,l}), \{k,l\} \neq \{i,j\} \}$. Thus, $I$ has to be of 
		even cardinal so that $k=2l$ for some $l \geq 1$. There are $(2l-1)!! = (2l-1) (2l-3) \dots 1$ 
		permutations that are product of transpositions since they are in bijection with 
		pairings of $2l$ elements. Each such permutation gives a contribution of 
		$\left( \E[a_{1,2}a_{2,1}] \right)^{l} = t^{l} $ and has a signature of $(-1)^{l}$. Therefore, 
			\begin{align*}
			\E [\det(A_I)] &= (2l-1)!! (-t)^{l},
		\end{align*}
		which only depends on the cardinal of $I$. Thus, 
		\begin{equation*}
			\E[S_{2l}^{(n)} ] =  \binom{n}{2l} (2l-1)!! (-t)^{l} 
		\end{equation*}
	\end{proof}
	
	\begin{proof}[Proof of Proposition \ref{prop:mean_poly}]
		Let $z \in \D$. Using \eqref{eq:secular_coefs_dev}, Lemma \ref{lem:expectation_secular} and 
		$\frac{(2k-1)!!}{(2k)!} = \frac{1}{2^{k} k!}$,
		\begin{align*}
			\E [f_{n,t}(z)] &=  \ed^{-ntz^2/2}  \sum_{k=0}^{\lfloor n/2 \rfloor} 
			\binom{n}{2k} (1+tz^2)^{n-2k} \frac{z^{2k}}{n^{k}}(2k-1)!! (-t)^{k} \\
			 &=  \ed^{-ntz^2/2}  \left( \sqrt{\frac{t}{n}} z \right)^{n} \sum_{k=0}^{\lfloor n/2 \rfloor}
			  \frac{(-1)^{k} n!}{2^{k} (n-2k)! k!} \left( \sqrt{\frac{n}{t}} \left( \frac{1}{z} + tz \right) \right)^{n -2k } \\
			&=  \ed^{-ntz^2/2}  \left( \sqrt{\frac{t}{n}} z \right)^{n} 
			H_n \left( \sqrt{\frac{n}{t}} \left( \frac{1}{z} + tz \right) \right).
		\end{align*}
	\end{proof}

\noindent
Note that from Proposition \ref{prop:mean_poly}, 
one can derive the result of Theorem \ref{th:cv_mean_poly_univ} 
by using uniform asymptotics of Hermite polynomials. 
In particular, this approach does not require the use 
of Proposition \ref{proposition:uniform_control}. 
Since this alternative method involves more intricate computations, we opted to prove 
Theorem \ref{th:cv_mean_poly_univ} using the approach presented in the beginning of 
this section.

\printbibliography

\end{document}